\documentclass[11pt]{article}
\usepackage[utf8]{inputenc}
\usepackage[top=2.3cm, bottom=2.3cm,left=2.5cm, right=2.5cm]{geometry}
\usepackage[maxbibnames=99,backend=biber,
style=numeric-comp,giveninits=true,date=year,
doi=false,isbn=false,url=false
]{biblatex}
\usepackage{float}
\usepackage{graphicx}
\usepackage{outlines}
\usepackage{amsfonts}
\usepackage{amsmath}
\usepackage{amsthm} 
\usepackage{amssymb}
\usepackage{xfrac} 
\usepackage{nicefrac}
\usepackage{array}
\usepackage{tabulary}
\usepackage{comment}
\usepackage{pdfpages}
\usepackage{url}
\usepackage{comment}
\usepackage{pdfpages}
\usepackage{lscape}
\usepackage{parskip}
\usepackage{tikz-cd}
\tikzcdset{arrow style=tikz, diagrams={>=stealth}}
\usetikzlibrary{calc,shapes.geometric} 
\usepackage[font={footnotesize}]{caption}
\usepackage[font={footnotesize}]{subcaption}
\usepackage{titlesec}
\usepackage[hidelinks]{hyperref}
\usepackage{listings}
\usepackage{mdframed}
\usepackage{xcolor}
\usepackage{mathtools}
\usepackage{booktabs}
\usepackage[plainruled]{algorithm2e}
\usepackage{setspace}
\usepackage{tocloft} 
\usepackage{csquotes} 
\usepackage{epigraph}
\usepackage{indentfirst}
\usepackage{prettyref}
\newcommand{\pref}{\prettyref}
\newrefformat{fig}{Figure~\ref{#1}}
\newrefformat{tab}{Table~\ref{#1}}
\newrefformat{def}{Definition~\ref{#1}}
\newrefformat{prop}{Proposition~\ref{#1}}
\newrefformat{cor}{Corollary~\ref{#1}}
\newrefformat{rem}{Remark~\ref{#1}}
\newrefformat{obs}{Observation~\ref{#1}}
\newrefformat{ex}{Example~\ref{#1}}
\newrefformat{lem}{Lemma~\ref{#1}}
\newrefformat{sec}{\S\ref{#1}}
\newrefformat{subsec}{\S\ref{#1}}
\usepackage[colorinlistoftodos, textsize=small]{todonotes} 
\usepackage{enumitem}
\setlist[itemize]{topsep=0.3em,leftmargin=1.9em}
\setlist[enumerate]{topsep=0.2em,leftmargin=2.0em}
\newcommand\Item[1][]{%
  \ifx\relax#1\relax  \item \else \item[#1] \fi
  \abovedisplayskip=0pt\abovedisplayshortskip=0pt~\vspace*{-\baselineskip}}

\newcommand{\NN}{\mathbb{N}} 
\newcommand{\ZZ}{\mathbb{Z}} 
\newcommand{\QQ}{\mathbb{Q}} 
\newcommand{\RR}{\mathbb{R}} 
\newcommand{\CC}{\mathbb{C}} 
\newcommand{\TT}{\mathbb{T}} 
\newcommand{\DD}{\mathbb{D}} 
\newcommand{\HH}{\mathbb{H}} 

\newcommand{\shortminus}{\scalebox{0.7}[1.0]{\( - \)}}

\providecommand{\keywords}[1]{\noindent\small\textbf{\textit{Keywords---}} #1} 
\theoremstyle{plain}
\newtheorem{theorem}{Theorem}[section]
\newtheorem{lemma}[theorem]{Lemma}
\newtheorem{proposition}[theorem]{Proposition}

\newtheorem{innercustomgeneric}{\customgenericname}
\providecommand{\customgenericname}{}
\newcommand{\newcustomtheorem}[2]{%
  \newenvironment{#1}[1]
  {%
   \renewcommand\customgenericname{#2}%
   \renewcommand\theinnercustomgeneric{##1}%
   \innercustomgeneric
  }
  {\endinnercustomgeneric}
}
\newcustomtheorem{customthm}{Theorem}
\newcustomtheorem{customcor}{Corollary}
\theoremstyle{definition}
\newtheorem{definition}[theorem]{Definition}
\newtheorem*{definition*}{Definition}
\newtheorem{remark}[theorem]{Remark}
\newtheorem*{remark*}{Remark}

\newtheorem*{observation*}{Observation}

\newcustomtheorem{customdef}{Definition}

\newcustomtheorem{customproof}{Proof}
\usepackage{fancyhdr}
\numberwithin{equation}{section}
\makeatletter
    \def\@algocf@capt@plainruled{above}
    \renewcommand{\algocf@caption@plainruled}{%
    \vskip\AlCapSkip%
    \box\algocf@capbox%
    \vskip 8\algoheightrule}
    \titleformat{\section}{\Large\bfseries}{\thesection}{0.5em}{}
\makeatother
\DontPrintSemicolon

\begin{document}
\pagenumbering{arabic}
\title{\vspace{-5.0mm}\LARGE\scshape On the parameter space of fibered \\ hyperbolic polynomials \vspace{-2.00mm}}
\author{Robert Florido\thanks{Corresponding author: \href{mailto:robert.florido@ub.edu}{\textit{robert.florido@ub.edu}}} $^1$, Núria Fagella\thanks{This work is supported by the grants PID2023-147252NB-I00 and CEX2020-001084-M (Mar\'ia de Maeztu Excellence program) from the Spanish State Research Agency, and ICREA Acad\`emia 2020 from the Catalan government.} $^{1,2}$}
\date{\vspace{-1.5mm}
{\small $^1$Departament de Matem\`atiques i Inform\`atica, Universitat de Barcelona, Barcelona, Spain}\\%
{\small $^2$Centre de Recerca Matem\`atica, Barcelona, Spain}\\[2ex]%
December 18, 2024} 
\maketitle
\begin{abstract}
We present an application of quasiconformal (QC) surgery for holomorphic maps fibered over an irrational rotation of the unit circle, also known as \textit{quasiperiodically forced (QPF) maps}. It consists of modifying the fibered multiplier of an attracting invariant curve for a QPF hyperbolic polynomial. This is the analogue of the classical change of multiplier of an attracting cycle in the one-dimensional iteration case, to parametrize hyperbolic components of the Mandelbrot set.

Our goal is to show that, for a family of QPF quadratic maps with Diophantine frequency, the \textit{fibered multiplier map} associated to its unique attracting invariant curve, as a complex counterpart of the Lyapunov exponent, is a holomorphic submersion on a complex Banach manifold.
\end{abstract}

%
\renewcommand{\contentsname}{
    \large 
    Contents
}
\titleformat*{\section}{\Large\bfseries} 


\section{Introduction}
\label{sec:1_Intro}

Quasiperiodically forced (QPF) systems have been of special interest to modern research in nonlinear dynamics, particularly due to the emergence of strange non-chaotic attractors; see \cite{Feudel2006,Jager2009} for comprehensive surveys in the real setting. They naturally arise from the \textit{stroboscopic map} (a Poincaré map but using a section in time) for flows driven by a quasiperiodic forcing which depends on several incommensurate frequencies.

Here we study the complexification of QPF maps as \textit{fibered holomorphic maps} over an irrational rotation on the $1$-torus $\TT^1:=\RR/\ZZ$, in the following sense.

\begin{definition}[QPF maps]
    \label{def:QPFmap}
    We say that a continuous function $F:\TT^1\times\CC\to\TT^1\times\CC$ is a \textit{QPF (holomorphic) map}, with \textit{frequency} $\alpha\notin\QQ$, if it is a skew-product of the form
    \begin{equation}
        \label{eq:QPFsystem}
        F(\theta,z) = \big(\theta+\alpha,f_\theta(z)\big)
    \end{equation}
    such that $\mathcal{R}_\alpha:=\Pi_1\circ F$ is an irrational rotation of angle $\alpha$, and the \textit{fiber maps} $f_\theta:=\Pi_2\circ F$ are holomorphic for all $\theta\in\TT^1$, where $\Pi_j$ stands for the projection map to the $j$-th coordinate, $j\in\{1,2\}$.    
\end{definition}

We are interested in the long-term behavior of points under iteration of $F$, as a fibered holomorphic map over $\mathcal{R}_\alpha$ on $\TT^1$. Denote by $F^n:=(\mathcal{R}_{n\alpha},f_\theta^n)$ its $n$-th iterate, that is, given an initial point $(\theta,z)\in\TT^1\times\CC$,
\begin{equation}
    F^n(\theta,z) = \big(\theta+n\alpha, f_\theta^n(z)\big), \qquad \mbox{ where } \qquad f_\theta^n := f_{\theta+(n-1)\alpha} \circ \cdots \circ f_\theta
\end{equation}
is a forward composition of holomorphic maps. For ease of notation, we set $\CC_\theta:=\{\theta\}\times\CC$ so $F^n(\CC_\theta)\subset\CC_{\theta+n\alpha}$, and for any $\mathcal{U}\subsetneq\TT^1\times\CC$, we denote by $\mathcal{U}_\theta:=\Pi_2(\mathcal{U}\cap\CC_\theta)$ the \textit{fiber} of $\mathcal{U}$ at $\theta\in\TT^1$.

It is clear that the fibered map $F$ has no fixed or periodic points due to the minimality of the irrational rotation on the base, hence the simplest invariant objects under $F$ shall be defined as follows.

\begin{definition}[Invariant curves]
    Let $F=(\mathcal{R}_\alpha,f_\theta)$ be a QPF map. A continuous function $\gamma:\TT^1\to\CC$ is called an \textit{invariant curve} of $F$ if, for each $\theta\in\TT^1$, it satisfies
    \begin{equation}
        F(\theta,\gamma(\theta)) = \left(\theta+\alpha, \gamma(\theta+\alpha)\right), \qquad \mbox{ i.e. } \qquad f_\theta(\gamma(\theta)) = \gamma(\theta+\alpha).
    \end{equation}
\end{definition}
    
We may also say that $\gamma$ is a \textit{$p$-periodic curve} of $F$ if it is invariant under $F^p$, i.e. its \textit{graph}
\begin{equation}
    \left\{ (\theta,\gamma(\theta)): \theta\in\TT^1 \right\} 
\end{equation}
is forward invariant by $F^p$, where $p\geq 1$ is the smallest natural with this property. As the class of QPF holomorphic maps is closed under composition, it is enough to study the dynamics about invariant curves. 

Ponce \cite{Ponce2007phd} described the local dynamics of $F$ on average in a neighborhood of an invariant curve $\gamma$ assuming that $F$ is injective on it, that is, the differential $f_\theta':=\partial_z f_\theta$ does not vanish there. This is accomplished in terms of the sign of the \textit{Lyapunov exponent} of $\gamma$, defined as usual as the real number
\begin{equation}
    \Lambda(\gamma):= \int_{\TT^1} \mathrm{log}\left| f_\theta'(\gamma(\theta)) \right| d\theta,
\end{equation}
which averages the exponential rate of separation under iteration of initial points infinitesimally close to $\gamma$.

Here we consider the complex version of the Lyapunov exponent so that its imaginary part measures the average speed of rotation about $\gamma$ of nearby points under iteration (see \cite{Herman1983, Ponce2007phd}), taking into account the winding number of the curve $\theta\mapsto f_\theta'(\gamma(\theta))$ with respect to the origin. Furthermore, as for any sufficiently small constant displacement $h\in\CC$ with respect to the invariant curve $\gamma$,
\begin{equation}
    f_\theta^{}\left(\gamma(\theta)+h\right) - \gamma(\theta+\alpha) = f_\theta'\left(\gamma(\theta)\right) h + \mathcal{O}(|h|^2),
\end{equation}
we see that the rate of rotation and stretching close to $\gamma$, as well as the number of times that the image under $F$ of the nearby closed curve $\gamma+h$ winds around $\gamma$, is mainly given by the linear part of our QPF system.

In this manner, we can define a natural counterpart to the multiplier of a fixed point in the autonomous case as follows (cf. \cite[\S1.3]{Ponce2007phd}). Recall that the \textit{winding number} of a closed curve $\Gamma:\TT^1\to\CC$ about $z_0\notin \Gamma(\TT^1)$, denoted by $\mathrm{wind}(\Gamma,z_0)$, represents the number of turns that $\Gamma$ does counterclockwise around $z_0$.

\begin{definition}[Index and fibered multiplier]
    \label{def:IndexMultiplier}
    Let $F$ be a QPF map with an invariant curve $\gamma$, and $f_\theta$ its fiber map such that $f_\theta'(\gamma(\theta))\neq 0$ for all $\theta\in\TT^1$. The \textit{index} of $\gamma$ is the integer $m(\gamma):=\mathrm{wind}(f_\theta'(\gamma(\theta)),0)$, and
    \begin{equation}
        \kappa(\gamma):= \exp\left[ \int_{\TT^1} \mathrm{log}\left(e^{-2\pi i \hspace{0.2mm} m(\gamma)\theta} f_\theta'\left(\gamma(\theta)\right) \right) d\theta \right]\in\CC^*
    \end{equation}
    is the \textit{fibered multiplier} of $\gamma$, for any branch of the logarithm. The quantities $\Lambda(\gamma)\in\RR$ and $\rho(\gamma)\in\TT^1$ so that
    \begin{equation}
        \kappa(\gamma)=e^{\Lambda(\gamma)+2\pi i \rho(\gamma)},
    \end{equation}
    are known as the \textit{Lyapunov exponent} and the \textit{fibered rotation number} of $\gamma$, respectively.
\end{definition}

In this setting, the invariant curve $\gamma$ is said to be \textit{attracting}, \textit{indifferent}, or \textit{repelling} if $|\kappa(\gamma)|$ is less than, equal to, or greater than $1$, respectively. Fibered analogues of the main theorems on the local dynamics near fixed points in the autonomous case (see e.g. \cite{Milnor2006}) can be found in \cite{Ponce2007} via the following notion of conjugacy.

\begin{definition}[Fibered conjugacies]
    Let $\mathcal{U}$ and $\mathcal{V}$ be domains in $\TT^1\times\CC$. We say that two QPF maps $F:\mathcal{U}\to \mathcal{U}$ and $G:\mathcal{V}\to \mathcal{V}$ are \textit{topologically conjugate up to angle translation} $\nu\in\TT^1$ if there is a continuous map $H:\mathcal{U}\to \mathcal{V}$ of the form $H(\theta,z)=(\theta+\nu,h_\theta(z))$, where $h_\theta$ is a homeomorphism for all $\theta\in\TT^1$, such that
    \begin{equation}
        G \circ H = H \circ F.
    \end{equation}
    If $\nu=0$ we just say that $F$ and $G$ are \textit{conjugate} via $H$. Moreover, if $h_\theta$ can be chosen to be quasiconformal (resp. affine), $F$ is said to be \textit{quasiconformally (resp. affinely) conjugate} to $G$.
\end{definition}

We shall also use the terms \textit{fibered quasiconformal (resp. affine) map} for $H$ in relation to its fiber maps. These conjugacies correspond to continuous changes of coordinates (up to a shift in the angular variable) between QPF maps which share the same qualitative dynamics, relating their iterates as in \pref{fig:Conjugacy}. The degree of regularity of the conjugacy map $h_\theta$ determines the strength of the similarity between such dynamical systems. It is known that the Lyapunov exponent of $\gamma$ is preserved under conformal conjugacies, but its fibered rotation number does not in general, unless we consider appropriate conjugacies (see \pref{lem:InvMultiplier}). 

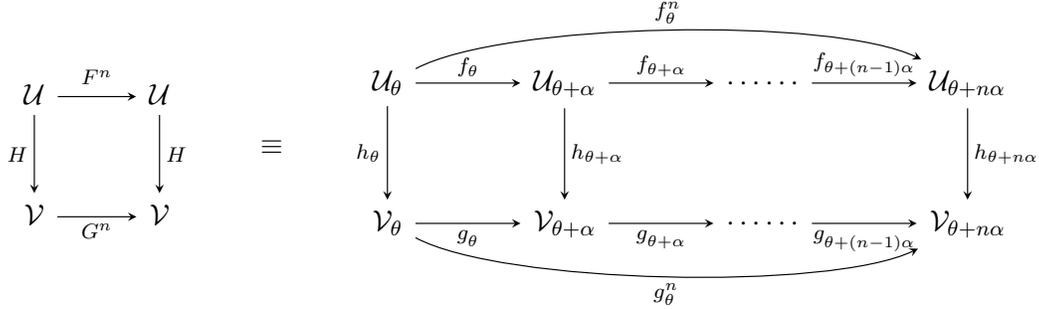
\begin{figure}[H]
    \centering
    \begin{tikzcd}[column sep=30pt,row sep=30pt]
    \mathcal{U} \arrow[r,"F^n_{}"] \arrow[d,"H"'] &
    \mathcal{U} \arrow[d,"H"]
    \\
    \mathcal{V} \arrow[r,"G^n_{}"'] &
    \mathcal{V}
    \end{tikzcd}
    $\qquad \equiv \qquad$ 
    \begin{tikzcd}[column sep=40pt,row sep=35pt]
    \mathcal{U}_\theta \arrow[r,"f_\theta"] \arrow[d,"h_\theta"']
    \arrow[bend left=27, looseness=0.5]{rrr}{f_\theta^n} &
    \mathcal{U}_{\theta+\alpha} \arrow[r,"f_{\theta+\alpha}"] \arrow[d,"h_{\theta+\alpha}"] &
    \cdots\cdots \arrow[r,"f_{\theta+(n-1)\alpha}"] &
    \mathcal{U}_{\theta+n\alpha} \arrow[d,"h_{\theta+n\alpha}"]
    \\
    \mathcal{V}_{\theta} \arrow[r,"g_{\theta}^{}"'] \arrow[bend right=27, looseness=0.5, swap]{rrr}{g_{\theta}^n} &
    \mathcal{V}_{\theta+\alpha} \arrow[r,"g_{\theta+\alpha}^{}"'] &
    \cdots\cdots \arrow[r,"g_{\theta+(n-1)\alpha}^{}"'] &
    \mathcal{V}_{\theta+n\alpha}
    \end{tikzcd}
    \caption{Commutative diagram on the conjugacy between the $n$-th iterates of two QPF maps $F$ and $G$ (with frequency $\alpha$) via the fibered map $H:=(\mathrm{Id},h_\theta)$, and its fiber-by-fiber version with $\mathcal{U}_\theta=\Pi_2(\mathcal{U}\cap \CC_\theta)$ and $\mathcal{V}_\theta=\Pi_2(\mathcal{V}\cap \CC_\theta)$, $\theta\in\TT^1$.}
    \label{fig:Conjugacy}
\end{figure}

In this article we aim to generalize the  \textit{soft surgery} procedure of changing the multiplier of an attracting cycle in the autonomous case, introduced by Sullivan in 1981. That is, we are going to modify the fibered multiplier of an attracting invariant curve $\gamma$, by deforming a collection of almost complex structures in an invariant way via successive changes of coordinates, starting from a linear normal form (see \pref{subsec:2_1_Konigs} and \pref{subsec:2_3_QC_MRMT}). Our construction can be locally defined on the connected component of the \textit{basin of attraction} of $\gamma$,
\begin{equation}
    \label{eq:basinAttraction}
    \mathcal{A}(\gamma) := \left\{ (\theta,z)\in\TT^1\times\CC: \ \left|f_{\theta}^n(z) - \gamma(\theta+n\alpha)\right|\to 0 \quad \mbox{ as } \quad n\to\infty  \right\}, 
\end{equation} 
which contains the invariant curve itself. A key ingredient to obtain a global quasiconformal (QC) map is the zero Lebesgue measure of the fractal boundary of $\mathcal{A}(\gamma)$, which motivates the restriction to QPF hyperbolic polynomials in the sense of Sester \cite{Sester1999}, as the central class of study in our work (see \pref{subsec:2_2_QPFpolyn} - \ref{subsec:2_3_QC_MRMT} for definitions).

\begin{customthm}{A}[Changing the fibered multiplier]
    \label{thm:A_QCqpf}
    Let $P$ be a QPF hyperbolic polynomial of Diophantine frequency $\alpha$, with an attracting invariant curve $\gamma$ of fibered multiplier $\kappa_0\in\DD^*$. Then, for any $\kappa\in\DD^*$, there is a QPF hyperbolic polynomial $P_\kappa$ and a fibered QC map $\Phi_\kappa:=(\mathrm{Id},\phi_{\kappa,\theta})$, appropriately normalized so that
    \begin{enumerate}
        \item[1.] $\gamma_\kappa(\theta):=\phi_{\kappa,\theta}^{}(\gamma(\theta))$ is an attracting invariant curve of $P_\kappa$ with fibered multiplier $\kappa$;
        \item[2.] $P$ and $P_\kappa$ are globally QC conjugate via $\Phi_\kappa$, and $\Phi_\kappa$ is conformal off the basin of attraction of $\gamma$.
    \end{enumerate}
    Moreover, for any simply-connected neighborhood $U_0$ of $\kappa_0$ in $\DD^*$, we have that
    \begin{enumerate}
        \item[3.] For each $z\in\CC$, $\phi_{\kappa,\theta}^{}(z)$ is continuous in $\theta\in\TT^1$ and holomorphic in $\kappa\in U_0$; 
        \item[4.] The fiber maps $p_{\kappa,\theta}^{}:=\phi_{\kappa,\theta+\alpha}^{} \circ p_{\theta}^{} \circ \phi_{\kappa,\theta}^{}$ depend holomorphically on $\kappa\in U_0$.
    \end{enumerate}
\end{customthm}

The QC surgery construction involved in the proof of this result (see \pref{sec:3_ProofThmA}) relies on applying fiberwise a version of the Measurable Riemann Mapping Theorem (MRMT) depending on parameters; see \pref{subsec:2_3_QC_MRMT}, especially \pref{thm:MRMTfibers}. The resulting integrating maps, globally defined in our hyperbolic polynomial setting (see also \pref{rem:localQCconj}), depend continuously on the fiber $\theta$ (as a real parameter) and analytically on the parameter $\kappa$.

In fact, our QC procedure will provide further information on the fibered multiplier for quadratic families of QPF hyperbolic polynomials. We can reduce our analysis (see \pref{sec:4_App}) to QPF quadratic maps of the form
\begin{equation}
    \label{eq:QuadraticQlambdaINTRO}
    Q_\lambda(\theta,z) := \left( \theta+\alpha, \ z^2 + \lambda(\theta) z \right), \qquad \mbox{where} \qquad \lambda\in\mathcal{C}^*(\TT^1).
\end{equation} 
Here $\mathcal{C}^*(\TT^1)$ denotes the space of non-vanishing continuous functions on $\TT^1$, as an open subset of the infinite-dimensional Banach space $\left(\mathcal{C}(\TT^1), \lVert \cdot \rVert_\infty\right)$ of continuous loops in $\CC$, with the sup-norm $||\lambda||_\infty:=\sup_{\theta\in\TT^1} |\lambda(\theta)|$.

Following Sester \cite[\S5]{Sester1999} we denote by $\mathcal{M}(\TT^1)\subset\mathcal{C}(\TT^1)$ the connectedness locus of the family of QPF quadratic polynomials, parametrized as in \pref{eq:QuadraticQlambdaINTRO}, and by $\mathcal{M}_0(\TT^1)\subset\mathcal{M}(\TT^1)$ the open subset of hyperbolic parameters generalizing the main cardioid of the Mandelbrot set (see \pref{lem:Qfamily}). For our purposes, we let $\mathcal{M}_0^*(\TT^1):=\mathcal{M}_0(\TT^1)\cap \mathcal{C}^*(\TT^1)$, and focus on
\begin{equation}
    \label{eq:MandelbrotQ0}
    \mathcal{H}_0^*(\TT^1) := \left\{ \lambda\in\mathcal{M}_0^*(\TT^1): \ \int_{\TT^1} \log|\lambda(\theta)|d\theta < 0, \quad \mbox{and} \quad \mathrm{wind}(\lambda(\theta),0)=0 \right\},
\end{equation}
which is indeed associated to those QPF hyperbolic quadratic maps $Q_\lambda$ with a unique attracting invariant curve, $\gamma\equiv 0$, of zero-index, whose basin of attraction is a connected open set in $\TT^1\times\CC$. Note that it is a complex Banach manifold since $\mathcal{M}_0(\TT^1)$ is known to be open in $\mathcal{C}(\TT^1)$, and both the Lyapunov exponent and the winding number are continuous maps when restricted to the space $\mathcal{C}^*(\TT^1)$.

The change of the fibered multiplier $\kappa_0\in\DD^*$ of $\gamma$ is achieved by means of \pref{thm:A_QCqpf} which, together with an appropriate normalization of the integrating maps, enables to track the relation between hyperbolic parameters for this QPF family in terms of the targeted fibered multiplier $\kappa\in U_0$. In the non-fibered case, this type of QC surgery reveals that the multiplier map, associated to an attracting cycle of the standard quadratic polynomial, is a conformal isomorphism from a hyperbolic component of the Mandelbrot set to $\DD$.

Here we show that the \textit{fibered multiplier map} $\widehat{\kappa}:\mathcal{H}_0^*(\TT^1)\to\DD^*$, defined as 
\begin{equation}
    \label{eq:MuliplierMapQ}
    \widehat{\kappa}(\lambda):= \exp\left( \int_{\TT^1} \mathrm{log} \hspace{0.3mm} \lambda(\theta) \hspace{0.3mm} d\theta \right)
\end{equation}
and associated to the unique attracting invariant curve of $Q_\lambda$ in the zero-index case, is a holomorphic submersion in the context of complex Banach manifolds (see \pref{subsec:2_4_Banach}). Sometimes it may be convenient to consider the mapping $\widehat{\chi}$ defined from $\mathcal{H}_0^*(\TT^1)$ onto the left half-plane $\HH_\ell$ such that $\widehat{\kappa} = \exp \circ \widehat{\chi}$, assigning to each parameter $\lambda$ the complex version of the Lyapunov exponent of $\gamma$.

\begin{customthm}{B}[Fibered multiplier map -- quadratic family]
    \label{thm:B_QCapp}
    Let $\mathcal{Q}_\lambda$ be the family of QPF quadratic polynomials 
    \begin{equation}
        Q_\lambda(\theta,z) = \big(\theta+\alpha, \ z^2 + \lambda(\theta) z \big), 
    \end{equation}
    with Diophantine frequency $\alpha$. Then the fibered multiplier map $\widehat{\kappa}:\mathcal{H}_0^*(\TT^1)\to\DD^*$ is a holomorphic submersion.
\end{customthm}

This provides a holomorphic foliation of $\mathcal{H}_0^*(\TT^1)$ of codimension $1$, i.e. a partition of it into submanifolds of constant dimension, whose leaves, the level sets $\{\widehat{\kappa}=\kappa_0\}$, are holomorphic graphs over $\DD^*$; see \pref{thm:HoloSubmersions}. These foliations, or more generally \textit{laminations} (i.e. foliations of closed subsets of complex Banach manifolds), play a central role in complex dynamics as, e.g., holomorphic motions are codimension-one laminations (see more in \cite{Lyubich2024}). Other parameter regions may be explored under some additional regularity on $\lambda(\theta)$.

We believe that the parameter space of $\mathcal{Q}_\lambda$ and the foliation by the fibered  multiplier map $\widehat{\kappa}$ are worth further investigation in our QPF holomorphic setting, which is, however, out of the scope of this paper.

\noindent\textbf{Notation.} \ For a set $A\subset X\in\{ \CC,\TT^1\times\CC \}$, we denote by $\mathrm{int\hspace{0.3mm}} A$, $\overline{A}$, $\partial A$ and $A^c:=X\backslash A$, its interior, closure, boundary and complement in $X$, resp., and for a function $F:X\to X$, we set $F^{-1}(A):=\{ x\in X: F(x)\in A \}$. Given the Euclidean (product) metric on $X$ and $\varepsilon > 0$, denote the \textit{$\varepsilon$-neighborhood} (or \textit{$\varepsilon$-fattening}) of $A$ by
\begin{equation}
    A^\varepsilon := \big\{ x\in X: \ \operatorname{dist}(x,A):=\inf_{a\in A} \operatorname{dist}(x,a) <\varepsilon \big\}.
\end{equation}
In particular, denote by $T^\varepsilon(\theta_0,z_0):=\{\theta_0\}^\varepsilon \times \{z_0\}^\varepsilon = (\theta_0-\varepsilon,\theta_0+\varepsilon)\times B_\varepsilon(z_0)$ the \textit{open $\varepsilon$-tube} about $(\theta_0,z_0)\in\TT^1\times\CC$, where $B_\varepsilon(z_0):=\{ z: |z-z_0|<\varepsilon \}$ is the open ball of radius $\varepsilon$ centered at $z_0$; let $\mathbb{D}:=B_1(0)$. 

\noindent Recall that $\CC_\theta:=\{\theta\}\times\CC$, and for $\mathcal{A}\subset\TT^1\times\CC$, we write $\mathcal{A}=\bigsqcup_{\theta\in\TT^1} \{\theta\}\times\mathcal{A}_\theta$, where $\mathcal{A}_\theta:=\Pi_2(\mathcal{A}\cap\CC_\theta)$ is the \textit{fiber} of $\mathcal{A}$ at $\theta$. For a fibered map $F:\TT^1\times\CC\to\TT^1\times\CC$, we use the lowercase letter $f_\theta$ for its fiber map, and denote by $\mathrm{Crit}(F)$ the \textit{critical set} of $F$, i.e. the points in its domain where the derivative $f_\theta'$ vanishes. By abuse of notation, a curve $\gamma$ will refer both to the continuous function $\gamma:\TT^1\to\CC$ and its graph in $\TT^1\times\CC$, while $F(\gamma)$ will indicate the image curve $\widetilde{\gamma}:\TT^1\to\CC$ under the fibered map $F$, given by $\widetilde{\gamma}(\theta):=f_\theta(\gamma(\theta))$.

\noindent\textbf{Acknowledgments.} \ We are greatful to Gustavo R. Ferreira for valuable comments and insights on the theory of Banach manifolds. We also thank Igsyl Dom\'inguez and Mario Ponce for interesting discussions.

\section{Preliminaries} 
\label{sec:2_Pre}

The local dynamics of QPF holomorphic maps around an invariant curve $\gamma$ is determined to a large extent by the fibered multiplier $\kappa(\gamma)=e^{\Lambda(\gamma)+2\pi i \rho(\gamma)}$ in the sense of \pref{def:IndexMultiplier}. It is well-known that its Lyapunov exponent $\Lambda(\gamma)$ is preserved by fibered conformal conjugacies (i.e. one-to-one analytic on each fiber), however this is not always the case for the fibered rotation number $\rho(\gamma)$ as remarked in \cite[\S1.3]{Ponce2007phd}. We first prove the following lemma on this matter that capitalizes on the notion of fibered conjugacy in our QPF setting. 

\begin{lemma}[Invariance of the fibered multiplier]
    \label{lem:InvMultiplier}
    Let $F$ be a QPF map with an invariant curve $\gamma$ of index $m(\gamma)$ and fibered multiplier $\kappa(\gamma)$. Consider any fibered conformal map $H(\theta,z)=(\theta+\nu,h_\theta(z))$, $\nu\in\TT^1$, and let $\eta:=\mathrm{wind}(h_\theta'(\gamma(\theta)),0)$. Denote by $\widetilde{\gamma}$ the corresponding invariant curve of $\widetilde{F}:=H\circ F\circ H^{-1}$. Then
    \begin{enumerate}
        \item $m(\widetilde{\gamma})=m(\gamma)$;
        \item $\Lambda(\widetilde{\gamma})=\Lambda(\gamma)$; and
        \item $\rho(\widetilde{\gamma})=\rho(\gamma) + \eta\alpha - m(\gamma)\nu \quad (\mathrm{mod} \ 1)$.
    \end{enumerate}
\end{lemma}

\begin{proof}
    Define $\widetilde{F}:=H\circ F\circ H^{-1}$ as the QPF map in the new coordinates, which takes the form
    \begin{equation*}
        \widetilde{F}(\varphi,u) = ( \varphi+\alpha, \widetilde{f}_{\varphi}(u)), \qquad \mbox{where} \qquad \widetilde{f}_{\varphi}:=h_{\varphi-\nu+\alpha} \circ f_{\varphi-\nu} \circ h_{\varphi-\nu}^{-1},
    \end{equation*}
    and $\alpha\notin\QQ$ is the frequency of $F$, for any $(\varphi,u)\in\TT^1\times\CC$. Given the invariant curve $\gamma$ of $F$, it is clear that 
    \begin{equation*}
        \widetilde{\gamma}(\varphi):=h_{\varphi-\nu}\big(\gamma(\varphi-\nu)\big)
    \end{equation*}
    is an invariant curve of $\widetilde{F}$, i.e. $\widetilde{f}_{\varphi}(\widetilde{\gamma}(\varphi)) = \widetilde{\gamma}(\varphi+\alpha)$. 
    By the chain rule and the Inverse Function Theorem,
    \begin{equation*}
    \widetilde{f}_{\varphi}'(\widetilde{\gamma}(\varphi))= \frac{h_{\varphi-\nu+\alpha}'\big(\gamma(\varphi-\nu+\alpha)\big)}{h_{\varphi-\nu}'\big(\gamma(\varphi-\nu)\big)} \ f_{\varphi-\nu}'\left(\gamma(\varphi-\nu)\right).
    \end{equation*}
    Let $m(\widetilde{\gamma}):=\mathrm{wind}\big(\widetilde{f}_{\varphi}'(\widetilde{\gamma}(\varphi)),0\big)$ be the \textit{index} of $\widetilde{\gamma}$, which coincides with that of $\gamma$ since the winding number is invariant under orientation-preserving homeomorphisms on $\CC$, known to be isotopic to the identity (see e.g. \cite{McCoy2012}, or \cite{Fagella2019} in our QPF context). Thus, we can compute the fibered multiplier of $\widetilde{\gamma}$ as
   \begin{equation*}
       \begin{split}
       \int_{\TT^1} \mathrm{log}\left(e^{-2\pi i \hspace{0.2mm} m(\widetilde{\gamma})\varphi} \hspace{0.4mm} \widetilde{f}_{\varphi}'(\widetilde{\gamma}(\varphi)) \right) d\varphi & = 
       \int_{\TT^1} \mathrm{log}\left(e^{-2\pi i m(\gamma)\nu} e^{-2\pi i \hspace{0.2mm} m(\gamma)\theta} \hspace{0.2mm} f_{\theta}'\left(\gamma(\theta) \right) \right) d\theta 
       \\
       & + \int_{\TT^1} \mathrm{log}\left(\frac{e^{-2\pi i \hspace{0.2mm} \eta (\theta+\alpha)} \hspace{0.2mm} h_{\theta+\alpha}'\left(\gamma(\theta+\alpha)\right)}{e^{-2\pi i \hspace{0.2mm} \eta \theta} \hspace{0.2mm} h_{\theta}'\left(\gamma(\theta)\right)} \ e^{2\pi i \hspace{0.2mm} \eta\alpha}  \right) d\theta,
       \end{split}
   \end{equation*} 
   by means of the change of angular variable $\varphi=\theta+\nu$. Therefore, exponentiating both sides and using the properties of the logarithm, we obtain that $\kappa(\widetilde{\gamma})=e^{2\pi i (\eta\alpha - m(\gamma)\nu)} \kappa(\gamma)$, or equivalently,
    \begin{equation*}
        \Lambda(\widetilde{\gamma}) = \Lambda(\gamma), \qquad \mbox{and} \qquad \rho(\widetilde{\gamma}) = \rho(\gamma) + \eta\alpha - m(\gamma)\nu \quad (\mathrm{mod} \ 1),
    \end{equation*}
    as claimed, considering that $\kappa(\widetilde{\gamma})=e^{\Lambda(\widetilde{\gamma})+2\pi i \rho(\widetilde{\gamma})}$.
    $\hfill\square$ 
\end{proof}

Given such a fibered change of coordinates $H$, defined in a neighborhood of an invariant curve $\gamma$, the integer $\eta=\mathrm{wind}(h_\theta'(\gamma(\theta)),0)$ is often called the \textit{index} of $H$ on $\gamma$. In particular, it follows from this lemma that fibered conjugacies of \textit{zero-index} ($\eta=0$) keeping the angle fixed (i.e. $\nu=0$), always preserve the fibered multiplier of invariant curves. These will be indeed the type of conjugacies involved in our main construction.

\subsection{Linearization around attracting invariant curves}
\label{subsec:2_1_Konigs}

Here we focus on the linear normal form of a QPF holomorphic map $F$ about an attracting invariant curve, generalizing Koenigs Linearization Theorem  \cite{Carleson1993} near an attracting fixed point in the non-fibered case. 

Given an invariant curve $\gamma$ of $F$ with Lyapunov exponent $\Lambda(\gamma)<0$, Ponce \cite[\S2]{Ponce2007phd} proved the existence of an open \textit{invariant tube} $\mathcal{T}(\gamma)\subset\TT^1\times\CC$ about $\gamma$ such that $F(\mathcal{T}(\gamma))\subset \mathcal{T}(\gamma)$ and, for all $\theta$, $\mathcal{T}_\theta(\gamma):=\Pi_2(\mathcal{T}(\gamma)\cap\CC_\theta)$ is conformally isomorphic to $\DD$, and contained in the the basin of attraction $\mathcal{A}(\gamma)$ of $\gamma$, defined in \pref{eq:basinAttraction}; thus
\begin{equation}
    \mathcal{A}(\gamma) = \bigcup_{n\geq 0} F^{-n}(\mathcal{T}(\gamma)).
\end{equation}
In fact, $F$ was shown to be \textit{weakly linearizable} in $\mathcal{T}(\gamma)$, i.e. conformally conjugate to a QPF contracting map 
\begin{equation}
    (\theta,z) \mapsto \left(\theta+\alpha, \ a(\theta) \hspace{0.4mm} z \right), \qquad \mbox{with} \qquad a(\TT^1)\subset \DD^*.
\end{equation}

Even further, in the case of a Diophantine driving frequency $\alpha$, $F$ can be shown to be \textit{strongly linearizable} in the sense of the theorem below, which follows directly from \cite[Prop.~2]{Fagella2019}, together with \cite[Prop.~3.1]{Ponce2007}. Recall that the class $\mathcal{D}$ of \textit{Diophantine numbers} consists of those $\alpha\notin\QQ$ such that, for some $\delta>0$ and $\tau\geq 2$,
\begin{equation}
    \left| \alpha- \frac{p}{q}\right| > \frac{\delta}{|q|^{\tau}}
\end{equation}
for all $\sfrac{p}{q}\in\QQ$. We say that $\alpha$ is \textit{Diophantine of type} $\tau$ and write $\alpha\in\mathcal{D}_\tau$; e.g. the golden mean $\frac{\sqrt{5}-1}{2}\in \mathcal{D}_2$. This arithmetic property helps to deal with the problem of the so-called \enquote{small divisors} when solving the cohomological equations involved, in connection to the theory of exponential stability (see e.g. \cite{Yoccoz1992}). 

\begin{theorem}[Strong linearization]
    \label{thm:Konigs}
    Let $F$ be a QPF map of frequency $\alpha\in\mathcal{D}$, with an attracting invariant curve $\gamma$ of index $m$, fibered multiplier $\kappa\in\DD^*$, and tube $\mathcal{T}(\gamma)$ as above. Then $F$ is conformally conjugate to
    \begin{equation}
        \label{eq:KonigsM}
        L(\theta,z):=(\theta+\alpha, \ \kappa \hspace{0.3mm} e^{2\pi i m \theta}\hspace{0.2mm} z),
    \end{equation}
    via $\Psi:=(\mathrm{Id},\psi_\theta):\mathcal{T}(\gamma)\to\TT^1\times\DD$, which can be chosen of zero-index with $\psi_\theta(\gamma(\theta))=0$ for all $\theta$.
\end{theorem}

In the remainder of this subsection, and for the sake of completeness, we point out the main steps to get such a linearizing map $\Psi$ (see details in \cite{Fagella2019,Ponce2007}), which makes the following diagram commute:
\begin{figure}[H]
    \centering
    \begin{tikzcd}[column sep=30pt,row sep=30pt]
    \mathcal{T}(\gamma)\arrow[r,"F"] \arrow[d,"\Psi"'] &
    \mathcal{T}(\gamma) \arrow[d,"\Psi"]
    \\
    \TT^1\times\DD \arrow[r,"L"'] &
    \TT^1\times\DD
    \end{tikzcd}
\end{figure}

\noindent Note that $\gamma$ may be placed at the zero-section $\TT^1\times \{0\}$ via the fibered translation $(\theta,z)\mapsto (\theta,z-\gamma(\theta))$. Thus, without loss of generality, assume that the invariant curve of $F$ under consideration is $\gamma\equiv 0$, and
\begin{equation}
    F(\theta,z) = \left(\theta+\alpha, \  a(\theta)z+z^2 b_\theta(z)\right),
\end{equation}
where $a(\theta):=f_\theta'(0)\neq 0$, and $b:\TT^1\times\CC\to\CC$ is a continuous function which is holomorphic in each fiber. Given that $m=\mathrm{wind}(a(\theta),0)$, there exists a holomorphic branch $\mathcal{L}(\theta)$ of $\log\left(e^{-2\pi i m\theta} a(\theta)\right)$ on $\TT^1$. In the situation of \pref{thm:Konigs}, $\gamma\equiv 0$ has Lyapunov  exponent $\Lambda:=\log|\kappa|<0$. Write $\kappa=e^{\Lambda+2\pi i \rho}$, with $\rho\in\TT^1$. 
    
\noindent Following \cite[\S3.1]{Fagella2019}, the first step is to transform the linear part of $F$ into $\kappa \hspace{0.2mm} e^{2\pi i m \theta}$ via the fibered linear map $\Psi_1(\theta,z):=(\theta, e^{u(\theta)} z)$, where $u:\TT^1\to\CC$ is taken as the zero-mean solution of the cohomological equation
\begin{equation}
    \label{eq:CohomologicalEq}
    u(\theta+\alpha) - u(\theta) = \Lambda + 2\pi i \rho - \mathcal{L}(\theta)
\end{equation}
such  that
\begin{equation}
    \widetilde{F}(\theta,z) := \Psi_1\circ F\circ\Psi_1^{-1}(\theta,z) = \big(\theta+\alpha,  \ \kappa \hspace{0.3mm} e^{2\pi i m\theta} z + z^2 \hspace{0.3mm}\widetilde{b}_\theta(z)\big), 
\end{equation}
where $\widetilde{b}_\theta(z)= e^{u(\theta+\alpha)-2u(\theta)} b_\theta(z)$. This is possible since, comparing coefficients of the Fourier expansions from \pref{eq:CohomologicalEq} and denoting the $k$-th Fourier coefficient of $\mathcal{L}(\theta)$ by $\mathcal{L}_k$, we obtain the series
\begin{equation}
    u(\theta) = \sum_{k\in\ZZ^*} \frac{\mathcal{L}_k}{1-e^{2\pi i k\alpha}} e^{2\pi i k \theta},
\end{equation}
which is absolutely and uniformly convergent, as its coefficients shall decay exponentially with the Fourier mode $|k|$ due to the Diophantine condition on $\alpha$. 
    
\noindent Moreover, if $M:=\sup_{\theta\in\TT^1} \big\{ \widetilde{b}_\theta(z): z\in\overline{\DD} \big\}$ and $R_c:=\operatorname{dist}\big(0, \mathrm{Crit}(\widetilde{F})\big)$, the invariant tube can be given as
\begin{equation}
    \mathcal{T}(0) = \Psi_1^{-1}\left( \TT^1\times B_{R}(0) \right), \qquad \mbox{with} \qquad R:=\min\left\{1, R_c, \frac{1-e^{\Lambda}}{2M} \right\},
\end{equation}
so that $\big|\widetilde{f}_\theta^n(z)\big| \leq \big(\frac{1+e^{\Lambda}}{2}\big)^n|z|$ for $|z|<R$, $n\in\NN$. The rest of the construction is an adaption of the classical one in the non-fibered setting due to Koenigs. In fact, it is shown in \cite[\S4.1]{Ponce2007} that the sequence
\begin{equation}
    \bigg\{ g_\theta^n(z):=\bigg( \prod_{j=0}^{n-1} \kappa \hspace{0.3mm} e^{2\pi i m (\theta+j\alpha)} \bigg)^{-1} \widetilde{f}_\theta^n(z) \bigg\}_{n\geq 0}
\end{equation}
converges uniformly on $\TT^1\times B_R(0)$ to a limit function $\widetilde{g}$ which is continuous and holomorphic in each fiber, with $\widetilde{g}_\theta^{}(0)=0$ and $\widetilde{g}_\theta'(0)=1$ for all $\theta$. Given that, for all $\theta\in\TT^1$ and $|z|<R$,
\begin{equation}
    g_{\theta+\alpha}^n \circ \widetilde{f}_\theta(z) = \kappa \hspace{0.2mm} e^{2\pi i m\theta} g_\theta^{n+1}(z),
\end{equation}
defining $\Psi_2(\theta,z):=(\theta, \widetilde{g}_\theta^{}(z))$, we conclude that $\Psi:=\Psi_2\circ\Psi_1$ is a conformal conjugacy between $F$ and the map $L$ in \pref{eq:KonigsM} from linearizing about $\gamma\equiv 0$, up to rescaling by $z\mapsto \frac{1}{R}z$. 
    
This derivation also clarifies why the linearizing map $\Psi$ can be (and will be from now on) chosen of zero-index, i.e. with $\mathrm{wind}(\psi_\theta'(\gamma(\theta)),0)=0$, since the winding number of a pointwise product of loops in $\CC^*$ is the sum of their winding numbers (see e.g. \cite[\S7]{Beardon1979}). Indeed, notice that in our case both $\Psi_1$ and $\Psi_2$ are of index zero, given that $\mathrm{wind}(e^{u(\theta)},0)=\frac{u(1)-u(0)}{2\pi i} = 0$ and $\mathrm{wind}(\widetilde{g}_\theta'(0),0)=0$, respectively.

\subsection{Fibered Julia sets of QPF polynomials}
\label{subsec:2_2_QPFpolyn}

In order to talk about the global dynamics of fibered holomorphic maps, we focus on the class of \textit{QPF polynomials} of fixed degree $d\geq 2$, i.e. continuous maps which take the form
\begin{equation}
    P(\theta,z)=\left(\theta+\alpha, \ p_\theta^{}(z)\right), \qquad \mbox{with} \qquad p_\theta(z) = c_d(\theta) z^d + \cdots + c_0(\theta),
\end{equation}
where $\alpha\notin\QQ$ and each coefficient $c_j:\TT^1\to \CC$ is continuous, with $c_d(\theta)\neq 0$ for all $\theta\in\TT^1$. 

In this polynomial context, there is a clear generalized notion of filled-in Julia set as detailed in the work of Sester \cite{Sester1997phd}. Note that, analogously to the one-dimensional polynomial case, we can find an \textit{escape radius}
\begin{equation}
    R^*:= \max_{\theta\in\TT^1} \left(1, \frac{1+|c_{d-1}(\theta)|+\dots+|c_0(\theta)|}{|c_d(\theta)|} \right),
\end{equation}
so that $|p_\theta^{}(z)|\geq \frac{|z|^d}{R^*}$ for $|z|>R^*$, defining the \textit{basin of attraction of $\infty$} as the non-empty open connected set
\begin{equation}
    \mathcal{A}(\infty) := \big\{(\theta,z): p^n_\theta(z)\to\infty \quad \mbox{ as } \quad n\to\infty \big\} = \bigcup_{n\in\NN} f^{-n}\left(\TT^1 \times \overline{D(0,R^{*})}^{\hspace{0.3mm} c} \right).
\end{equation}

\begin{definition}[Fibered filled-in Julia set]
    Let $P$ be a QPF polynomial. The \textit{fibered filled-in Julia set} of $P$ consists of those points in $\TT^1\times\CC$ with bounded orbit under iteration by $P$, denoted by
    \begin{equation}
        \mathcal{K}:=\big\{ (\theta,z)\in\TT^1\times\CC: \ \sup_{n\in\NN} |p_\theta^n(z)|<\infty \big\}.
    \end{equation}
\end{definition}

It is known that the set $\mathcal{K}=A(\infty)^c$ is full and compact, just like all its fibers $\mathcal{K}_\theta:=\Pi_2(\mathcal{K}\cap\CC_\theta)$ (see details in \cite[Prop.~2.3]{Sester1999}), and it is also \textit{completely invariant} in the sense that, for all $\theta\in\TT^1$ and $n\in\NN$,
\begin{equation}
    p_\theta^n(\mathcal{K}_\theta) = \mathcal{K}_{\theta+n \alpha}, \qquad p_\theta^{-n}(\mathcal{K}_\theta) = \mathcal{K}_{\theta -n \alpha}.
\end{equation}

Furthermore, Sester \cite{Sester1997phd} defined the \textit{fibered Julia set} of $P$, where the chaotic behavior occurs, as
\begin{equation}
    \mathcal{J} := \overline{\bigsqcup_{\theta\in\TT^1}\{\theta \}\times \partial \mathcal{K}_\theta} \subset \partial \mathcal{K},
\end{equation}
and proved that, although this inclusion is generally strict, equality holds for fibered hyperbolic polynomials.

\begin{definition}[QPF hyperbolic polynomials]
    We say that a QPF polynomial $P$ is \textit{hyperbolic} if it is uniformly expanding on its fibered Julia set, i.e. there exists $A>0$ and $\sigma>1$ such that, for all $(\theta,z)\in \mathcal{J}$,
    \begin{equation}
        \left|(p_\theta^n)'(z)\right| \geq A \sigma^n.
    \end{equation}
\end{definition}

Among the several equivalent definitions of hyperbolicity \cite[Thm.~1.1]{Sester1999}, the following one is the most useful for our purposes: the \textit{postcritical set} of a hyperbolic polynomial $P$ do not intersect the Julia set, i.e.
\begin{equation}
    \label{eq:PostCritical}
    \overline{\bigcup_{n\in\NN} P^n\left(\mathrm{Crit}(P)\right)} \hspace{0.2mm} \cap \hspace{0.2mm} \mathcal{J} = \emptyset, \qquad \mbox{where} \qquad \mathrm{Crit}(P):= \left\{ (\theta,z)\in\TT^1\times\CC: \ p_\theta'(z)=0 \right\}.
\end{equation}
In fact, in analogy to the non-fibered case, the set $\mathrm{Crit}(P)\subset\TT^1\times\CC$ of critical curves of an arbitrary QPF polynomial $P$ determines to a large extent its dynamics since, e.g., it is known that all fibers of its filled-in Julia set $\mathcal{K}$ are connected if and only if $\mathrm{Crit}(P)\subset \mathcal{K}$, and the basin of attraction of any attracting invariant curve of $P$ must intersect $\mathrm{Crit}(P)$ at infinitely many fibers (and quite often); see more details in \cite[\S3]{Sester1999}. 

In the hyperbolic setting, the set-valued maps $\theta\mapsto\mathcal{J}_\theta$ and $\theta\mapsto \mathcal{K}_\theta$ are known to be \textit{Hausdorff continuous}; see \cite[Prop.~4.1]{Sester1999} for $\mathcal{J}$, and \cite[Prop.~3.4]{Dominguez2023} for $\mathcal{K}$ in our QPF context. This is in analogy to the non-fibered case \cite{Douady1994}, as here the angular variable (the fiber) plays the role of the complex parameter. Recall that this notion of continuity for maps from $\TT^1$ to the set $\mathrm{Comp}^*(\CC)$ of non-empty compact subsets of $\CC$, is induced by the so-called \textit{Pompeiu-Hausdorff distance} (see e.g. \cite{Rockafellar1998}) between two sets $A,B\in \mathrm{Comp}^*(\CC)$:
\begin{equation}
    \operatorname{dist}_H(A,B) := \max\Big\{ \sup_{a\in A} \operatorname{dist}(a,B), \ \sup_{b\in B} \operatorname{dist}(A,b) \Big\} = \inf\left\{ \varepsilon>0: A\subset B^\varepsilon, \ B\subset A^\varepsilon \right\}.
\end{equation}
The Hausdorff continuity for $\mathcal{J}$ and the Koebe Distortion Theorem lead to the following result on the measure of the Julia set in the hyperbolic case \cite[Cor.~4.3]{Sester1999}, which is key for our QC surgery construction.

\begin{proposition}[Fibered Julia set]
    \label{prop:JuliaMeas}
    Let $P$ be a QPF hyperbolic polynomial, and $\mathcal{K}$ its filled-in Julia set. 
    \noindent Then $\mathcal{J}=\partial \mathcal{K}$, and $\mathcal{J}_\theta=\partial\mathcal{K}_\theta$ which has zero Lebesgue measure for all $\theta\in\TT^1$.
\end{proposition}

Note that $\CC\backslash \mathcal{J}_\theta$ corresponds to the set of points $z$ such that the iterates $\{p_\theta^n\}_{n\in\NN}$ form a normal family on a neighborhood of $z$ (see remark in \cite[\S2.2]{Sester1999}, and \cite[\S5]{Roy2002}), as in the general non-autonomous framework \cite{Comerford2006}. Hence this notion and related properties may be fiberwise adapted to $\TT^1\times\CC$ in such a hyperbolic situation. 

\subsection{Quasiconformal maps and MRMT}
\label{subsec:2_3_QC_MRMT}

Recall that QC maps are generalizations of conformal (i.e. one-to-one analytic) maps which allow bounded distortion of angles. To be precise, a fiber map $h_\theta:\mathcal{U}_\theta\to\mathcal{V}_\theta$ is a \textit{$K$-QC} map, $K\geq 1$, if it is an orientation-preserving homeomorphism, absolutely continuous on lines and with bounded \textit{dilatation}:
\begin{equation}
    K_{\theta}:=\frac{1+\lVert \mu_\theta^{} \rVert_\infty}{1-\lVert \mu_\theta^{} \rVert_\infty} \leq K, \qquad \mbox{where} \qquad \lVert \mu_\theta^{} \rVert_\infty := \underset{z\in \mathcal{U}_\theta}{\mathrm{ess \hspace{0.8mm} sup}} \hspace{0.8mm} |\mu_\theta^{}(z)| \leq \frac{K-1}{K+1}
\end{equation}
is the essential supremum of the \textit{Beltrami coefficient} of $h_\theta$, defined as the function $\mu_\theta^{}\in L^\infty(\mathcal{U}_\theta)$ satisfying 
\begin{equation}
   \partial_{\bar{z}} h_\theta(z) = \mu_\theta^{}(z)  \hspace{0.4mm} \partial_z h_\theta(z),  \qquad \mbox{for almost every (a.e.)} \ z\in \mathcal{U}_\theta.
\end{equation}
Geometrically, $\mu_\theta^{}$ can be interpreted as a measurable field of infinitesimal ellipses (defined up to scaling) on the tangent bundle over $\mathcal{U}_\theta^{}$, known as an \textit{almost complex structure}. At the points where $h_\theta$ is differentiable, its differential carries these ellipses to a field of circles (the \textit{standard complex structure}). In fact, $K_{\theta}$ measures the deviation of $h_\theta$ from the rigid class of conformal (i.e. $1$-QC) maps; see further details in \cite{Ahlfors2006, Astala2008, Branner2014, Lehto1973, Hubbard2006}.

The main tool in the quasiconformal surgery construction we shall perform to prove \pref{thm:A_QCqpf}, is the celebrated \textit{Measurable Riemann Mapping Theorem} (MRMT), also known as the \textit{Integrability Theorem}.

\begin{theorem}[MRMT -- global version]
    \label{thm:MRMTglobal}
    Let $\mu\in L^\infty(\CC)$ be a Beltrami coefficient satisfying $||\mu||_\infty\leq k<1$. Then there is a QC map $\phi:\CC\to\CC$ which \textit{integrates} $\mu$, i.e. solves the Beltrami equation 
    \begin{equation}
        \partial_{\bar{z}}\phi \hspace{0.3mm} / \hspace{0.3mm} \partial_z\phi = \mu 
    \end{equation}
    almost everywhere. Furthermore, $\phi$ is unique up to post-composition with a non-constant affine map.
\end{theorem}

For more details and versions, we refer to \cite{Ahlfors2006,Astala2008,Branner2014,Lehto1973,Hubbard2006}. From now on, unless otherwise stated, $\phi$ is said to be \textit{normalized} if $\phi(0)=0$ and $\phi(1)=1$. 

For our purposes, we need a generalized version of the MRMT which states that if a sequence of Beltrami coefficients converges pointwise a.e., then the sequence of integrating maps converges uniformly on compact subsets of $\CC$ to the QC map integrating the limiting Beltrami coefficient; see \cite[Lem.~5.3.5]{Astala2008}, and \cite{Ahlfors1960}. Here we adapt the statement to our fibered setting, with the analytic dependence on parameters \cite[Cor.~5.7.5]{Astala2008}.

\begin{theorem}[MRMT depending on parameters]
    \label{thm:MRMTfibers}
    Let $\{\mu_{\kappa,\theta}^{}\}_{\theta}^{}$ be a collection of Beltrami coefficients on $\CC$ such that $||\mu_{\kappa,\theta}||_\infty\leq k <1$ for $\kappa\in\DD$ and $\theta\in\TT^1$, and $\{\phi_{\kappa,\theta}^{}\}_\theta^{}$ be their normalized integrating maps. Then,
    \begin{enumerate}
        \item If $\theta\mapsto \mu_{\kappa,\theta}^{}(z)$ is continuous for a.e. $z\in\CC$, then $\theta\mapsto\phi_{\kappa,\theta}^{}(z)$ is continuous for each $z\in\CC$.
        \item If $\kappa\mapsto \mu_{\kappa,\theta}^{}(z)$ is holomorphic for every $z\in\CC$, then $\kappa\mapsto\phi_{\kappa,\theta}^{}(z)$ is holomorphic for each $z\in\CC$. 
    \end{enumerate}
\end{theorem}

The goal of the quasiconformal surgery procedure in \pref{sec:3_ProofThmA} is to change the fibered multiplier of an attracting invariant curve of a QPF hyperbolic model map by means of the Integrability Theorem and, as in the one-dimensional case, a fibered analogue of the notion of invariant Beltrami coefficients will be crucial.

Recall that, given a QC map $f:U\to V$ (which is known to be absolutely continuous with respect to the Lebesgue measure) and a Beltrami coefficient $\mu$ in $V$, the \textit{pullback} of $\mu$ under $f$, denoted by $f^*\mu$ in $U$ (using the notation in \cite{Branner2014}), represents the field of infinitesimal ellipses (at a.e. $u\in U$) which are transported by the differential of $f$ to the corresponding field of infinitesimal ellipses (at $f(u)\in V$) described by $\mu$. Provided that $\mu$ is defined on an open set containing $U\cup V$, $\mu$ is called $f$-\textit{invariant} if $f^*\mu(u)=\mu(u)$ for a.e. $u\in U$.

In analogy, for a QPF map $F=(R_\alpha,f_\theta)$, we say that the collection $\{\mu_{\kappa,\theta}^{} \}_{\theta}^{}$ of Beltrami coefficients on $\CC$ (which we call a \textit{fibered Beltrami coefficient} on $\TT^1\times\CC$) is $F$-\textit{invariant} if, for all $\theta\in\TT^1$,
\begin{equation}
    \label{eq:Fpullback}
    f_{\theta}^{*} \hspace{0.4mm} \mu_{\kappa,\theta+\alpha}^{}(z) = \mu_{\kappa,\theta}^{}(z) \qquad \mbox{for \ a.e.} \ z\in\CC.
\end{equation}

\subsection{Holomorphic maps on Banach spaces}
\label{subsec:2_4_Banach}

Our goal in \pref{thm:B_QCapp} is to show that the fibered multiplier map $\widehat{\kappa}$ for the QPF quadratic family $Q_\lambda$, given by \pref{eq:MuliplierMapQ}, is a holomorphic submersion. Hence we need to introduce the following notion of holomorphicity in the framework of Banach spaces (see \cite[App.~A]{Poschel1987}), which plays a central role in modern functional analysis, to extend classical results and techniques from finite-dimensional spaces to infinite dimensions. Recall that a \textit{complex Banach space} is a complete normed vector space over $\CC$, as e.g. the $n$-dimensional Euclidean space $\CC^n$, the space $\mathcal{C}(X)$ of continuous complex-valued functions on a compact set $X$ (equipped with the sup-norm), or the space $\mathcal{L}(X,Y)$ of bounded linear maps between Banach spaces with the operator norm. 

\begin{definition}[Holomorphicity]
    \label{def:HolomorphicBanach}
    Let $(M,\lVert\cdot\rVert_M^{})$ and $(N,\lVert\cdot\rVert_N^{})$ be complex Banach spaces, and $U\subset M$ an open subset. A continuous map $F:U\to N$ is \textit{holomorphic} (or \textit{complex-analytic}) if $F$ is of class $\mathcal{C}^1$ on $U$, i.e. for each $x^*\in U$ there exists a bounded linear map $DF(x^*): M\to N$, the \textit{derivative} of $F$ at $x^*$, such that
    \begin{equation}
        \frac{\lVert F(x^*+h) - F(x^*) - DF(x^*) \ h \rVert_N^{}}{\lVert h \rVert_M^{}} \longrightarrow 0, \qquad \mbox{as} \qquad h\to 0,
    \end{equation}
    and its derivative $DF:U\to \mathcal{L}(M,N)$ is continuous. It is said to be \textit{weakly holomorphic} on $U$ if, for each $x^*\in U$, $v\in M$ and $L\in \mathcal{L}(N,\CC)$, the complex-valued function
    \begin{equation}
        t \mapsto L\circ F (x^*+t v)
    \end{equation}
    is holomorphic in some neighborhood of the origin (in the usual sense of one complex variable).
\end{definition}

Many properties of holomorphic mappings in Banach spaces (e.g. the Open Mapping Theorem, Liouville's Theorem or the Maximum Principle) can be derived from the corresponding ones for holomorphic functions of one complex variable, both having local power series expansions. This is mainly due to the following result based on the generalized Cauchy's formula for Banach spaces (see e.g. \cite[App.~A5]{Hubbard2006} or  \cite[App.~A]{Poschel1987}, and also \cite{Mujica1986} for more details). Remarkably, a weakly holomorphic map is indeed holomorphic if, in addition, it is \textit{locally bounded}, i.e. bounded in some neighborhood of each point of its domains of definition.

\begin{theorem}[Characterization of holomorphic maps]
    \label{thm:WeaklyHoloPower}
    Let $F:M\to N$ be a map between complex Banach spaces, and $U\subset M$ an open subset. Then the following are equivalent:
    \begin{enumerate}
        \item[1.] $F$ is holomorphic on $U$.
        \item[2.] $F$ is locally bounded and weakly holomorphic on $U$.
        \item[3.] $F$ is of class $\mathcal{C}^\infty$ on $U$, and is locally represented by its Taylor series. 
    \end{enumerate}
\end{theorem}

Recall that the \textit{Taylor series} of $F$ at any $x^*\in U$ is the unique power series $\sum_{k=0}^\infty P_k(x-x^*)$ that converges uniformly to $F$ in some neighborhood of $x^*$, where each $P_k$ is a \textit{homogeneous polynomial of degree $k$} on $M$, i.e. there exists a continuous $k$-linear symmetric map $\mathcal{P}_k:M^k\to N$ such that $P_k(x)=\mathcal{P}_k(x,\dots,x)$.

Holomorphic mappings are not only studied on open subsets of Banach spaces but on more general global objects known as \textit{complex Banach manifolds}, i.e. complex-analytic manifolds locally modelled on open subsets of complex Banach spaces (via biholomorphic coordinate transformations). Indeed, any open subset $U$ of a Banach space is a manifold with the global chart $(U,\mathrm{Id}_U)$, and most of differential calculus carries over such as the Implicit and Inverse Function Theorems (see more in \cite[App.~A5]{Hubbard2006} or \cite[App.~B]{Poschel1987}, as well as \cite{Fritzsche2002}). 

We are particularly interested in \textit{holomorphic submersions} between complex Banach manifolds, that is, holomorphic maps $F:\mathcal{M}\to \mathcal{N}$ satisfying that, for each $x\in\mathcal{M}$, there is a chart $(U,\varphi)$ at $x$ and $(V,\psi)$ at $F(x)$ such that $\varphi$ gives a biholomorphism of $U$ onto a product $U_1\times U_2$ (where $U_1$ and $U_2$ are open in some Banach spaces), and $F$ in these coordinates, namely $\psi\circ F\circ \varphi^{-1}$, becomes a \textit{projection} (onto its first factor) near $x$. This means that the map factors as the canonical projection $U_1\times U_2\to U_1$ followed by a biholomorphism of $U_1$ onto an open subset of $V$; see \cite[\S II.1]{Lang1995} and \cite[\S1.6]{Nag1988}. Observe that it is the generalization of the standard submersion from $\CC^m$ to $\CC^n$ (with $m\geq n$), given by the projection onto the first $n$ coordinates. 

As a consequence of the Inverse Function Theorem (see \cite[\S I.5]{Lang1995}), such a submersion is by definition \textit{split} in the sense that, for each $x\in\mathcal{M}$, the differential $dF(x)\in\mathcal{L}(T_x\mathcal{M},T_{F(x)}\mathcal{N})$ is a surjective map between tangent spaces, and its kernel splits (i.e. admits a closed complementary subspace in $T_x\mathcal{M}$; which is always the case in finite dimensions \cite[\S II.1]{Lang1995}). Furthermore, it provides a local description for submanifolds (perhaps defined by analytic equations) as follows; see e.g. \cite[App.~A5]{Hubbard2006} or \cite[\S1.6]{Nag1988} (and also \cite[\S IV.1]{Fritzsche2002}). 

\begin{theorem}[Holomorphic submersions]
    \label{thm:HoloSubmersions}
    Let $F:\mathcal{M}\to\mathcal{N}$ be a holomorphic map between complex Banach manifolds, $x^*\in \mathcal{M}$, and $y^*:=F(x^*)$. Then the following are equivalent:
    \begin{enumerate}
        \item $F$ is a holomorphic submersion at $x^*$.
        \item $F$ has a local holomorphic section at $y^*$.
        \item There exist a neighborhood $\mathcal{U}\subset \mathcal{M}$ of $x^*$, a neighborhood $\mathcal{V}\subset\mathcal{N}$ of $y^*$ with $F(\mathcal{U})\subset \mathcal{V}$, a complex Banach manifold $\mathcal{Z}$ and a holomorphic map $G:\mathcal{U}\to \mathcal{Z}$ such that $\widetilde{F}:=(F,G)$ defines a biholomorphic map from $\mathcal{U}$ to an open subset of $\mathcal{V}\times\mathcal{Z}$.
    \end{enumerate}
\end{theorem}

Here by a local holomorphic \textit{section} for $F$ at $y^*$ we refer to a holomorphic \textit{right-inverse} of $F$ on some neighborhood $\mathcal{V}\subset \mathcal{N}$ of $y^*$, i.e. a holomorphic mapping $s:\mathcal{V}\to \mathcal{M}$ satisfying $s(y^*)=x^*$ and $F\circ s = \mathrm{Id}_{\mathcal{V}}$. Notice that the fiber $\mathcal{M}_{y^*}:=F^{-1}(\{y^*\})$ is a complex Banach submanifold of $\mathcal{M}$ of \textit{codimension} (defined as the dimension of the quotient $\mathcal{M}/\mathcal{M}_{y^*}$) equal to $n:=\dim\mathcal{N}\leq \dim\mathcal{M}$, since
\begin{equation}
    \mathcal{M}_{y^*}\cap \mathcal{U} = \widetilde{F}^{-1}\left(\{y^*\}\times\mathcal{Z}\right)\cap \mathcal{U}.
\end{equation}

Thus, such a holomorphic submersion ensures the existence of a \textit{holomorphic foliation} $\mathcal{F}$ of $\mathcal{M}$ of codimension $n$, that is, a partition of the complex manifold $\mathcal{M}$ into disjoint immersed submanifolds (called the \textit{leaves} of $\mathcal{F}$) of codimension $n$, locally given by the level sets $\{F=y^*\}$, $y^*\in\mathcal{N}$; see more details in \cite{Lyubich2024,LinsNeto2020}.

\section{Proof of Thm.~\ref{thm:A_QCqpf}}
\label{sec:3_ProofThmA} 

Consider a QPF hyperbolic polynomial $P=(\mathcal{R}_\alpha, p_\theta)$ of Diophantine frequency $\alpha$, with an attracting invariant curve $\gamma$ of index $m_0$ and fibered multiplier $\kappa_0\in\DD^*$. We will perform a surgery construction to get a QPF map $P_\kappa$ quasiconformally conjugate to $P$, with an invariant curve of a given fibered multiplier $\kappa\in\DD^*$. 

The idea of the proof is to define a collection of non-trivial Beltrami forms $\{\mu_{\kappa,\theta}\}_{\theta\in\TT^1}$ which are invariant under the fibered map $P$, in the sense of \pref{eq:Fpullback}. We shall later use the MRMT (see \pref{subsec:2_3_QC_MRMT}) to fiberwise integrate the corresponding complex structures, and conjugate $P$ by the resulting fibered integrating map to produce a new fibered holomorphic map with the desired properties. We now proceed to make this construction precise.

First, by \pref{thm:Konigs} there exists a fibered linearizing map $\Psi:\mathcal{T}(\gamma)\mapsto\TT^1\times\DD$, where $\mathcal{T}(\gamma)$ is an open invariant tube about $\gamma$, such that $P$ is conformally conjugate via $\Psi$ to the QPF linear contracting map
\begin{equation*}
    L(\theta,z):=(\theta+\alpha, \ \kappa_0 \hspace{0.2mm} e^{2\pi i m_0 \theta}\hspace{0.2mm} z).
\end{equation*}

Fix $\kappa\in\DD^*$, and write $\kappa=e^{\Lambda+2\pi i \rho}$ and $\kappa_0=e^{\Lambda_0+2\pi i \rho_0}$, where $\Lambda,\Lambda_0 < 0$ and $\rho,\rho_0 \in (\shortminus\frac{1}{2},\frac{1}{2}]$. Recall that our goal is to deliver a QPF polynomial $P_\kappa$ which is QC conjugate to $P$, with an attracting invariant curve $\gamma_\kappa$ of fibered multiplier $\kappa$ (and same index $m_0$ as $\gamma$, since the winding number is a topological invariant). Hence, the initial step is to find a (fibered) QC conjugacy between $L$ and
\begin{equation*}
    L_\kappa(\theta,z):=(\theta+\alpha, \ \kappa \hspace{0.2mm} e^{2\pi i m_0 \theta}\hspace{0.2mm} z).
\end{equation*} 
    
    \textup{\textbf{\textit{Step 1}.} \textit{Conjugating the linearized map $L$ to $L_\kappa$ via a fibered QC map. 
    }}

    \noindent\textup{For convenience we first find an explicit formula for it as a real-linear map on $\TT^1\times\HH_\ell$, where $\HH_\ell$ is the left half-plane. Denote by $\mathrm{Exp}(\theta,z):=(\theta,e^z)$ the \textit{fibered exponential map} (of period $2\pi i$), and let $\widetilde{\chi}^{},\widetilde{\chi}_\kappa^{}:\TT^1\to\HH_\ell$ be defined as
    \begin{equation*}
        \widetilde{\chi}^{}(\theta):=\Lambda_0+2\pi i (\rho_0+m_0\theta), \qquad \widetilde{\chi}_\kappa^{}(\theta):=\Lambda+2\pi i (\rho+m_0\theta),
    \end{equation*}
    so that $\chi_0^{}:=\widetilde{\chi}^{}(0)$ and $\chi_\kappa^{}:=\widetilde{\chi}_\kappa^{}(0)$, with imaginary parts in $(-\pi,\pi]$, satisfy that $e^{\chi_0^{}} = \kappa_0$, $e^{\chi_\kappa} = \kappa$. Let
    \begin{equation*}
        T_{\widetilde{\chi}}(\theta,\zeta) := \left( \theta+\alpha, \ \zeta+\widetilde{\chi}(\theta) \right),
    \end{equation*}
    be a QPF translation, and consider the fibered $\RR$-linear map $\widetilde{L}_\kappa:\TT^1\times\HH_\ell\to\TT^1\times\HH_\ell$ of the form
    \begin{equation*}
        \widetilde{L}_\kappa(\theta,\zeta) = \big(\theta, \widetilde{\ell}_{\kappa,\theta}^{}(\zeta)\big), \qquad \mbox{where} \qquad \widetilde{\ell}_{\kappa,\theta}(\zeta):= \widetilde{a}_\kappa(\theta) \zeta + \widetilde{b}_\kappa(\theta) \overline{\zeta}.
    \end{equation*}
    Here each $\widetilde{\ell}_{\kappa,\theta}^{}$ is normalized to send the basis $\{\widetilde{\chi}^{}(\theta),2\pi i\}$ onto the basis $\{\widetilde{\chi}_\kappa^{}(\theta),2\pi i\}$ (in analogy to the non-fibered example in \cite[\S~1.5]{Branner2014}), $\theta\in\TT^1$, so that the following diagram commutes:}
    \begin{figure}[H]
    \centering
    \begin{tikzcd}[column sep=40pt,row sep=35pt]
    \TT^1\times\HH_\ell \arrow[r,"T^{}_{\widetilde{\chi}^{}}"] \arrow[d,"\widetilde{L}_\kappa"'] &
    \TT^1\times\HH_\ell  \arrow[d,"\widetilde{L}_\kappa"]
    \\
    \TT^1\times\HH_\ell \arrow[r,"T^{}_{\widetilde{\chi}_\kappa^{}}"]  &
    \TT^1\times\HH_\ell 
    \end{tikzcd}
    \end{figure}

    \noindent\textup{It follows from this diagram ($\widetilde{L}_\kappa\circ T_{\widetilde{\chi}}=T_{\widetilde{\chi}_\kappa^{}} \circ \widetilde{L}_\kappa$), and the two normalization conditions on $\widetilde{L}_\kappa$ above, that
    \begin{align*}
        \widetilde{a}_\kappa(\theta+\alpha) &= \widetilde{a}_\kappa(\theta), &\qquad 
        \widetilde{a}_\kappa(\theta+\alpha)\widetilde{\chi}^{}(\theta) + \widetilde{b}_\kappa(\theta+\alpha) \overline{\widetilde{\chi}^{}(\theta)} &= \widetilde{\chi}_\kappa^{}(\theta), \\
        \widetilde{b}_\kappa(\theta+\alpha)&=\widetilde{b}_\kappa(\theta), & \widetilde{a}_\kappa(\theta)-\widetilde{b}_\kappa(\theta)&=1,
    \end{align*} 
    for all $\theta\in\TT^1$. These relations reveal that both coefficients of $\widetilde{\ell}_{\kappa,\theta}$ must be autonomous and equal to
    \begin{equation}
    \label{eq:Lcoefficients_ab}
    a_\kappa := \frac{\chi_\kappa^{}+\overline{\chi_0^{}}}{\chi_0^{}+\overline{\chi_0^{}}} = \frac{\mathrm{Log}\hspace{0.2mm}\kappa + \mathrm{Log}\hspace{0.2mm}\overline{\kappa_0^{}}}{2\hspace{0.2mm}\mathrm{Log}\hspace{0.2mm}|\kappa_0^{}|}, \qquad b_\kappa:=a_\kappa-1 = \frac{\mathrm{Log}\hspace{0.2mm}\kappa - \mathrm{Log}\hspace{0.2mm}\kappa_0^{}}{2\hspace{0.2mm}\mathrm{Log}\hspace{0.2mm}|\kappa_0^{}|},
    \end{equation} 
    respectively, where $\mathrm{Log}$ denotes the principal branch of the logarithm; say $\widetilde{L}_\kappa=(\mathrm{Id},\widetilde{\ell}_{\kappa})$ by dropping the subscript $\theta$, with $\widetilde{\ell}_\kappa(\zeta)=a_\kappa \zeta + b_\kappa \overline{\zeta}$. Recall that $\widetilde{\chi}(\theta)=\chi_0^{}+2\pi i m_0\theta$ and $\widetilde{\chi}_\kappa^{}(\theta)=\chi_\kappa^{}+2\pi i m_0\theta$.
    }
    
    \noindent\textup{Now by choosing the branch $\log$ that maps $\DD^*$ onto the half-strip $\left\{ t\hspace{0.2mm}\chi_0^{} + s \hspace{0.2mm}i: t>0, \hspace{0.2mm} s\in[0,2\pi) \right\}$, our $\RR$-linear map can be expressed on $\DD^*$ as
    \begin{equation*}
        \widetilde{\phi}_{\kappa}(z) := \exp \circ \hspace{0.4mm}\widetilde{\ell}_{\kappa} \circ \log (z) = z \hspace{0.3mm} e^{2b_\kappa\log|z|}.
    \end{equation*}
    Note that if $\rho=\rho_0$, i.e. $\kappa/\kappa_0=e^{\Lambda-\Lambda_0}$, and $K:=\Lambda_0/\Lambda\geq 1$, we obtain the standard $K$-QC \textit{radial stretching} $z\mapsto z |z|^{1/K-1}$. By construction, the fibered QC map $\widetilde{\Phi}_\kappa:=(\mathrm{Id}, \widetilde{\phi}_{\kappa})$ makes the following diagram commute:
    }
    \begin{figure}[H]
    \centering
    \begin{tikzcd}[column sep=40pt,row sep=35pt]
    \TT^1\times\DD^* \arrow[r,"L^{}"] \arrow[bend right=50, looseness=0.54,swap]{ddd}{\widetilde{\Phi}_\kappa} &
    \TT^1\times\DD^* \arrow[bend left=50, looseness=0.54]{ddd}{\widetilde{\Phi}_\kappa}
    \\
    \TT^1\times\HH_\ell \arrow[r,"T^{}_{\widetilde{\chi}^{}}"] \arrow[d,"\widetilde{L}_\kappa"'] \arrow[u,"\mathrm{Exp}"]  &
    \TT^1\times\HH_\ell  \arrow[d,"\widetilde{L}_\kappa"] \arrow[u,"\mathrm{Exp}"']
    \\
    \TT^1\times\HH_\ell \arrow[r,"T^{}_{\widetilde{\chi}_\kappa^{}}"] \arrow[d,"\mathrm{Exp}"'] &
    \TT^1\times\HH_\ell  \arrow[d,"\mathrm{Exp}"]
    \\
    \TT^1\times\DD^* \arrow[r,"L_\kappa"'] &
    \TT^1\times\DD^*
    \end{tikzcd}
    \end{figure}

    \textup{\textbf{\textit{Step 2}.} \textit{Defining a $P$-invariant collection of Beltrami coefficients on $\mathcal{A}(\gamma)$. 
    }}

    \noindent\textup{Given that the complex dilatation of the real-linear map $\widetilde{\ell}_{\kappa}$ is exactly $\sfrac{b_\kappa}{a_\kappa}$, we derive from this construction that each fiber map $\widetilde{\phi}_\kappa$ is a QC map fixing the origin and integrating the autonomous Beltrami coefficient 
    \begin{equation}
        \label{eq:AutonomousBeltrami2}
        \widetilde{\mu}_\kappa(z) := \frac{b_\kappa}{a_\kappa} \hspace{0.3mm}\frac{z}{\bar{z}}, \qquad \mbox{with} \qquad \widetilde{K}:=\frac{|a_\kappa|+|b_\kappa|}{|a_\kappa|-|b_\kappa|} = \max\left\{\frac{\Lambda}{\Lambda_0}, \frac{\Lambda_0}{\Lambda}\right\}
    \end{equation} 
    as its dilatation (i.e. the ratio of the major to the minor axis of the ellipses), setting $\widetilde{\mu}_\kappa(0)=0$. It follows from the diagram above that $L_\kappa = \widetilde{\Phi}_\kappa \circ L \circ \widetilde{\Phi}_\kappa^{-1}$ on $\TT^1\times\DD$, which is different from $L=(\mathcal{R}_\alpha,\ell_\theta)$ whenever $\kappa\neq\kappa_0$. Moreover, the collection $\{\widetilde{\mu}_\kappa\}_{\theta\in\TT^1}$ is $L$-invariant, i.e. $\ell_{\theta}^*\hspace{0.3mm}\widetilde{\mu}_\kappa = \widetilde{\mu}_\kappa$ for all $\theta$. Here we emphasize that, as indicated by dropping the subscript $\theta$, these Beltrami coefficients are the same for all fibers.
    }

    \noindent\textup{Recall that the basin of attraction $\mathcal{A}(\gamma)$ of the invariant curve $\gamma$ may be obtained by taking preimages of the open invariant tube $\mathcal{T}(\gamma)$ of $\gamma$ (see \pref{subsec:2_1_Konigs}) under the QPF polynomial $P$, that is, 
    \begin{equation*}
        \mathcal{A}(\gamma) = \bigcup_{n\in\NN} \mathcal{T}^n(\gamma), \qquad \mbox{where} \qquad \mathcal{T}^n(\gamma):=P^{-n}\left( \mathcal{T}(\gamma) \right).
    \end{equation*}
    Thus we can define recursively a new collection of Beltrami coefficients $\{\mu_{\kappa,\theta}\}_\theta^{}$ in $\CC$ by first pulling back $\{\widetilde{\mu}_\kappa\}_\theta^{}$ under the fibered linearizing map $\Psi=(\mathrm{Id},\psi_\theta):\mathcal{T}(\gamma)\to\TT^1\times\DD$, and subsequently spreading to the whole basin $\mathcal{A}(\gamma)$ by the dynamics of $P=(\mathcal{R}_\alpha, p_\theta)$ in the following sense (and setting $\mu_0\equiv 0$ elsewhere):
    \begin{equation}
        \label{eq:Spreading_Beltrami2}
        \mu_{\kappa,\theta} := \begin{cases}
            \psi_{\theta}^* \hspace{0.3mm} \widetilde{\mu}_{\kappa} & \mbox{on} \quad \mathcal{T}_\theta(\gamma) 
            \\
            (p_{\theta}^n)^* \hspace{0.3mm} \mu_{\kappa,\theta+n\alpha} & \mbox{on} \quad \mathcal{T}_\theta^n(\gamma)\hspace{0.3mm}\backslash \hspace{0.3mm} \mathcal{T}_\theta^{n-1}(\gamma), \ \mbox{  for every } n\geq 1,  \\
            \mu_0 & \mbox{on} \quad \CC\hspace{0.3mm}\backslash \hspace{0.3mm} \mathcal{A}_\theta(\gamma). 
        \end{cases}
    \end{equation}
    Here the subscript $\theta$ refers to the fiber at $\theta\in\TT^1$ of the corresponding set or map as usual, and $\mu_0$ denotes the standard Beltrami form (associated to a measurable field of circles). Hence, as all pullbacks are done under holomorphic maps, the fibered Beltrami coefficient $\{\mu_{\kappa,\theta} \}_\theta^{}$ is $P$-invariant by construction (see \pref{subsec:2_3_QC_MRMT}), and linked to almost complex structures over $\TT^1\times\CC$ with uniformly bounded dilatation $\widetilde{K}$, given by \pref{eq:AutonomousBeltrami2}.
    }

    \noindent\textup{Furthermore, given a simply-connected neighborhood $U_0$ of $\kappa_0 = e^{\chi_0^{}}$ in $\DD^*$, we choose the branch $\mathcal{L}:U_0\to\CC$ of the logarithm satisfying $\mathcal{L}(\kappa_0)=\chi_0^{}$ with $\mathrm{Im}\chi_0^{}\in (-\pi,\pi]$, so $\chi_\kappa^{}:=\mathcal{L}(\kappa)$ depends holomorphically on $\kappa$. 
    Then, since the coefficients $a_\kappa$ and $b_\kappa$, explicitly given by \pref{eq:Lcoefficients_ab}, are holomorphically depending on $\chi_\kappa^{}$, we deduce that the Beltrami form $\widetilde{\mu}_\kappa(z)=\frac{b_\kappa}{a_\kappa} \frac{z}{\overline{z}}$ varies holomorphically with respect to $\kappa\in U_0$. 
    }

    \noindent\textup{Now observe that the collection $\{\mu_{\kappa,\theta}\}_\theta^{}$ of Beltrami coefficients in $\CC$ was derived from $\widetilde{\mu}_\kappa$ by pulling back with the fibered linearizing map $\Psi:\mathcal{T}(\gamma)\to\TT^1\times\DD$, and spreading to all $\mathcal{A}(\gamma)$ by successive pullbacks via $P$, both being (fibered) holomorphic maps independent of $\kappa$. Since its extension by $\mu_{\kappa,\theta}(z)=0$ for $z\in\CC\hspace{0.3mm}\backslash\hspace{0.3mm}\mathcal{A}_\theta(\gamma)$ is clearly holomorphically dependent on $\kappa$, we conclude that for every $z\in\CC$, $\kappa\mapsto \mu_{\kappa,\theta}(z)$ is holomorphic in $\kappa\in U_0$. In addition, note that
    \begin{equation*}
        ||\mu_{\kappa,\theta}||_\infty = ||\widetilde{\mu}_{\kappa}||_\infty = \left| \frac{\mathcal{L}(\kappa)-\mathcal{L}(\kappa_0)}{\mathcal{L}(\kappa)+\overline{\mathcal{L}(\kappa_0)}}\right|,
    \end{equation*}
    which is a continuous real-valued function in $\kappa$ from $\overline{U_0}$ to $[0,1)$, so it is automatically bounded and its maximum value is attained, say $k<1$. We are ready to utilize the Integrability Theorem with parameters. \vspace{0.3cm}
    }

    \textup{\textbf{\textit{Step 3}.} \textit{Applying the MRMT to integrate $\{\mu_{\kappa,\theta}\}_\theta^{}$ fiberwise}.} 

    \noindent\textup{Notice that in \textit{Step 2} we produced a Beltrami coefficient $\mu_{\kappa,\theta}(z)$ on each $\CC_\theta$, defined recursively by \pref{eq:Spreading_Beltrami2}, with uniformly bounded dilatation in $\theta$, $\kappa$ and $z$. By the global MRMT (see \pref{thm:MRMTglobal}) there exists a unique integrating map $\phi_{\kappa,\theta}:\CC\to\CC$, normalized to fix $0$ and $1$, such that $\phi_{\kappa,\theta}^*\hspace{0.3mm}\mu_0 = \mu_{\kappa,\theta}$, or equivalently $\partial_{\bar{z}} \phi_{\kappa,\theta}=\mu_{\theta,\kappa} \partial_z \phi_{\kappa,\theta}$, a.e. in $\CC$. Observe that $\phi_{\kappa,\theta}$ is indeed conformal off $\mathcal{A}_\theta(\gamma)$ by Weyl's lemma. }
    
    \noindent\textup{For a.e. $z\in\CC$, we are going to show the continuity of $\theta\mapsto \mu_{\kappa,\theta}(z)$ at any given $\theta^*\in\TT^1$, by distinguishing three different cases. Recall that $\gamma$ is the invariant curve of $P$ under scrutiny, with fibered multiplier $\kappa_0\in\DD^*$.}
    
    \noindent\textup{\textit{Case 1: $z\in\mathcal{A}_{\theta^*}(\gamma)$}. \hspace{0.3mm} Let $N$ be the first iterate such that $\big(\theta^*+N\alpha,p_{\theta^*}^{N}(z)\big)\in \mathcal{T}(\gamma)$. Consider $\varepsilon>0$ so that the $\varepsilon$-tube $T^{\varepsilon}\big(\theta^*+N\alpha,p_{\theta^*}^{N}(z)\big):=\{\theta^*+N\alpha\}^\varepsilon \times \{ p_{\theta^*}^{N}(z) \}^\varepsilon\subset \mathcal{T}(\gamma)$. Since the basin of attraction of $\gamma$ is open and $P$ is uniformly continuous in $\overline{\mathcal{A}(\gamma)}$, there exists $T^{\delta_\varepsilon}(\theta^*,z)\subset\mathcal{A}(\gamma)$ for some $\delta_\varepsilon>0$, satisfying
    \begin{equation*}
        P^{N}\big( T^{\delta_\varepsilon}(\theta^*,z) \big) \subset \mathcal{T}(\gamma). 
    \end{equation*}
    Hence, for any sequence $\{\theta_n\}_{n\in\NN}\subset \TT^1$ such that $\theta_n\to\theta^*$ as $n\to\infty$, it is clear from the recursive definition in \textit{Step 2}, together with the continuity of $p_\theta^N$ and $\psi_\theta$ in $\theta$, that $\mu_{\kappa,\theta_n}^{}(z)$ converges to $\mu_{\kappa,\theta^*}^{}(z)$
    as $n\to\infty$. 
    } 

    \noindent\textup{\textit{Case 2: $z\in\partial\mathcal{A}_{\theta^*}(\gamma)$}. \hspace{0.3mm} Recall that the hyperbolicity assumption implies that every fiber of the Julia set $\mathcal{J}=\partial\mathcal{K}$ for the QPF polynomial $P$ has zero Lebesgue measure (see \pref{prop:JuliaMeas}). Therefore, as $\partial\mathcal{A}_{\theta^*}(\gamma)\subset \mathcal{J}_{\theta^*}$, this situation does not require further analysis.}

    \noindent\textup{\textit{Case 3: $z\in\CC\hspace{0.2mm}\backslash\hspace{0.2mm}\overline{\mathcal{A}_{\theta^*}(\gamma)}$}. \hspace{0.3mm} Since $\overline{\mathcal{A}(\gamma)}^{\hspace{0.2mm}c}$ is an open set in $\TT^1\times\CC$, there is some $\delta>0$ such that $T^{\delta}(\theta^*,z)\subset \overline{\mathcal{A}(\gamma)}^{\hspace{0.2mm}c}$. Thus, the sequential continuity of $\theta\mapsto\mu_{\kappa,\theta}(z)$ at $\theta^*$ is here  straightforward, given that in \textit{Step 2} it was defined to be $\mu_0 \equiv 0$ at every point outside of the basin of attraction of $\gamma$.}

    \noindent\textup{In conclusion, the continuity of $\theta\mapsto \mu_{\kappa,\theta}(z)$ is guaranteed for a.e. $z\in\CC$, and $\kappa\mapsto \mu_{\kappa,\theta}(z)$ is holomorphic for every $z$, with $||\mu_{\kappa,\theta}||_\infty=k<1$ for all $\theta$ and $\kappa\in U_0$, as shown at the end of \textit{Step 2}. Hence, \pref{thm:MRMTfibers} delivers a fibered QC map $\Phi_\kappa=(\mathrm{Id},\phi_{\kappa,\theta})$ on $\TT^1\times\CC$ which is continuous in $\theta$ and holomorphic in $\kappa$, normalized in the standard way and fiberwise integrating the ($P$-invariant) collection of Beltrami forms $\{\mu_{\kappa,\theta}\}_\theta^{}$. \vspace{0.3cm}
    }
    
    \textup{\textbf{\textit{Step 4}.} \textit{Determining the new polynomial $P_\kappa$ with an invariant curve of fibered multiplier $\kappa$.}}

    \noindent\textup{First we see that
    \begin{equation*}
        P_\kappa(\theta,z) := \Phi_\kappa\circ P \circ \Phi_\kappa^{-1}(\theta,z) = \big(\theta+\alpha, \ \phi_{\kappa,\theta+\alpha}^{} \circ p_{\kappa,\theta}^{} \circ \phi_{\kappa,\theta}^{-1}(z)\big) 
    \end{equation*}
    is a QPF polynomial map of the same degree as $P$ so that, for any $\theta\in\TT^1$, the following diagram commutes:
    }
    \begin{figure}[H]
    \centering
    \begin{tikzcd}[column sep=30pt,row sep=30pt]
    (\CC, \ \mu_{\kappa,\theta}^{}) \arrow[r,"p_{\theta}^{}"] \arrow[d,"\phi_{\kappa,\theta}^{}"'] &
    (\CC,\ \mu_{\kappa,\theta+\alpha}) \arrow[d,"\phi_{\kappa,\theta+\alpha}^{}"]
    \\
    (\CC,\ \mu_{0}^{}) \arrow[r,"p_{\kappa,\theta}^{}"'] &
    (\CC,\ \mu_{0}^{})
    \end{tikzcd}
    \end{figure}

    \noindent\textup{This follows from the fact that each fiber map $p_{\kappa,\theta}$ in our construction is quasiconformal and preserves the standard complex structure, and so it is holomorphic by Weyl's lemma (which is actually the reason for building a $P$-invariant fibered Beltrami coefficient in \textit{Step 2}). Since $p_\theta^{}$ is a polynomial and each integrating map $\phi_{\kappa,\theta}$ is a homeomorphism on $\CC$, we deduce that $p_{\kappa,\theta}$ is also a polynomial (of the same degree). 
    }
    
    \noindent\textup{In order to show that $\Phi_\kappa$ is the fibered QC map we were looking for, we must verify that $\gamma_\kappa(\theta):=\phi_{\kappa,\theta}(\gamma(\theta))$ is an invariant curve of $P_\kappa$ with fibered multiplier $\kappa$. As $\gamma$ is invariant under $P$, it is clear from the commutative diagram above that $\gamma_\kappa$ is the corresponding invariant curve of $P_\kappa$.
    }
    
    \noindent\textup{For this purpose, we define the fibered map $\Psi_\kappa=(\mathrm{Id},\psi_{\kappa,\theta}^{}):\Phi_\kappa\left(\mathcal{T}(\gamma)\right)\to\TT^1\times\DD$ by
    \begin{equation*}
        \psi_{\kappa,\theta}^{}(z) := \widetilde{\phi}_\kappa\circ\psi_\theta^{}\circ \phi_{\kappa,\theta}^{-1}(z),
    \end{equation*}
    so that $\psi_{\kappa,\theta}^*\vspace{0.3mm}\mu_0=\mu_0$ for all $\theta\in\TT^1$, given that $\widetilde{\Phi}_\kappa$ and $\Phi_\kappa$ integrate $\{\widetilde{\mu}_{\kappa}\}_\theta$ and $\{\mu_{\kappa,\theta}\}_\theta$ fiberwise, respectively, and $\psi_{\theta}^*\vspace{0.3mm}\widetilde{\mu}_\kappa=\mu_{\kappa,\theta}$ by definition; see \pref{eq:Spreading_Beltrami2}. Recall that $\Psi=(\mathrm{Id},\psi_\theta^{})$ refers to the linearizing map on the invariant tube $\mathcal{T}(\gamma)$ of $\gamma$, and $\widetilde{\Phi}_\kappa=(\mathrm{Id},\widetilde{\phi}_{\kappa})$ is the (autonomous) QC map on $\TT^1\times\DD$ conjugating the linearized dynamics of $P$ and $P_\kappa$ about $\gamma$ and $\gamma_\kappa$, respectively, as described in \textit{Step 1}.
    }
    
    \noindent\textup{It follows from Weyl's lemma that $\Psi_\kappa$ is a fibered conformal map and, as it conjugates $P_\kappa$ about $\gamma_\kappa$ to the QPF linear map $L_\kappa(\theta,z)=(\theta+\alpha, \kappa e^{2\pi i m_0 \theta}z)$ about $\TT^1\times\{0\}$, we conclude that $\Psi_\kappa$ is indeed the linearizing map about $\gamma_\kappa$, i.e. $\kappa$ is the fibered multiplier of $\gamma_\kappa$. Note that we can consider the tubular neighborhood $\Phi_\kappa(\mathcal{T}(\gamma))$ as the open invariant tube $\mathcal{T}(\gamma_\kappa)$ of $\gamma_\kappa$.
    }
    
    \noindent\textup{We summarize our QC surgery construction in the following commutative diagram, where the fiber $\theta$ is preserved (resp. sent to $\theta+\alpha$) via  fibered maps along the vertical (resp. horizontal) direction: 
    }

    \begin{figure}[H]
    \centering
    \begin{tikzcd}[column sep=40pt,row sep=35pt]
    \big(\mathcal{T}(\gamma_\kappa), \ \{\mu_0^{}\}_\theta^{}\big) \arrow[r,"P_{\kappa}"] \arrow[bend right=63, looseness=0.6,swap]{ddd}{\Psi_{\kappa}} &
    \big(\mathcal{T}(\gamma_\kappa),\ \{\mu_0^{}\}_\theta^{}\big) \arrow[bend left=63, looseness=0.6]{ddd}{\Psi_{\kappa}}
    \\
    \big(\mathcal{T}(\gamma), \ \{\mu_{\kappa,\theta}^{}\}_\theta^{} \big) \arrow[r,"P"] \arrow[d,"\Psi"'] \arrow[u,"\Phi_\kappa"]  &
    \big(\mathcal{T}(\gamma), \ \{\mu_{\kappa,\theta}^{}\}_\theta^{} \big)  \arrow[d,"\Psi"] \arrow[u,"\Phi_\kappa"']
    \\
    \big(\TT^1\times\DD, \ \{\widetilde{\mu}_\kappa^{}\}_\theta^{}\big) \arrow[r,"L"] \arrow[d,"\widetilde{\Phi}_\kappa"'] &
    \big(\TT^1\times\DD, \ \{\widetilde{\mu}_\kappa^{}\}_\theta^{}\big)  \arrow[d,"\widetilde{\Phi}_\kappa"]
    \\
    \big(\TT^1\times\DD, \ \{\mu_0^{}\}_\theta^{}\big) \arrow[r,"L_\kappa"] &
    \big(\TT^1\times\DD, \ \{\mu_0^{}\}_\theta^{}\big)
    \end{tikzcd}
    \label{fig:QCsummary}
    \end{figure}

    \noindent\textup{Last, but not least important, we aim to show the holomorphic dependence of $P_\kappa=(\mathcal{R}_\alpha,p_{\kappa,\theta})$ on $\kappa$, in analogy to the non-fibered derivation in \cite[Lem. 1.40]{Branner2014}. Fix $\theta\in\TT^1$ and consider the given simply-connected neighborhood $U_0$ of $\kappa_0$ in $\DD^*$. From the commutative diagram above we have: $p_{\kappa,\theta}\circ \phi_{\kappa,\theta}(z) = \phi_{\kappa,\theta+\alpha} \circ p_{\theta}(z)$, for any $\kappa\in  U_0$ and $z\in\CC$.
    In order to stress the (holomorphic or antiholomorphic) dependence on both the parameter and the point, i.e. on the variables $(\kappa,\overline{\kappa},z,\overline{z})$, it may be written as 
    \begin{equation*}
        p_{\kappa,\theta}(\kappa,\overline{\kappa},\phi_{\kappa,\theta}(z),\overline{\phi_{\kappa,\theta}(z)}) = \phi_{\kappa,\theta+\alpha}(\kappa,\overline{\kappa},p_{\theta}(z),\overline{p_{\theta}(z)}).
    \end{equation*}}

    \noindent\textup{By differentiating both sides with respect to $\overline{\kappa}$ and applying the chain rule in the sense of distributions (as QC mappings can be approximated by $\mathcal{C}^{\infty}$-functions; see \cite[Cor. 5.5.8]{Astala2008}), we obtain that
    \begin{equation*}
        \frac{\partial\hspace{0.2mm} p_{\kappa,\theta}}{\partial \overline{\kappa}}\bigg|_{\phi_{\kappa,\theta}(z)} + \quad \frac{\partial \hspace{0.2mm} p_{\kappa,\theta}}{\partial z}\bigg|_{\phi_{\kappa,\theta}(z)} \cdot \ \frac{\partial\hspace{0.2mm} \phi_{\kappa,\theta}}{\partial \overline{\kappa}}\bigg|_{z} \quad + \quad \frac{\partial\hspace{0.2mm} p_{\kappa,\theta}}{\partial \overline{z}}\bigg|_{\phi_{\kappa,\theta}(z)} \cdot \ \frac{\partial \hspace{0.2mm}\overline{\phi_{\kappa,\theta}}}{\partial \overline{\kappa}}\bigg|_{z} \quad = \quad \frac{\partial \hspace{0.2mm}\phi_{\kappa,\theta+\alpha}}{\partial \overline{\kappa}}\bigg|_{p_{\theta}(z)}.
    \end{equation*}
    Therefore, since the polynomial $p_{\kappa,\theta}$ is holomorphic in $z\in\CC$ and the integrating map $\phi_{\kappa,\theta}$ is holomorphic in $\kappa\in U_0$ due to the MRMT depending on parameters (see conclusions of \textit{Steps 2} and \textit{3}), i.e. $\partial_{\bar{z}} p_{\kappa,\theta}\equiv 0$ and $\partial_{\bar{\kappa}} \phi_{\kappa,\theta}\equiv 0$, we see that $\frac{\partial\hspace{0.2mm} p_{\kappa,\theta}}{\partial \bar{\kappa}}\big|_{\phi_{\kappa,\theta}(z)}=0$, and so $p_{\kappa,\theta}$ varies holomorphically with respect to $\kappa\in U_0$.}

\begin{remark}[On the hyperbolicity assumption]
    \label{rem:localQCconj}
    Notice that the hypothesis of hyperbolicity of the QPF polynomial $P$ is the one that allows to give the fibered QC map $\Phi_\kappa$ as a global conjugacy between $P$ and $P_\kappa$, changing the fibered multiplier of the attracting invariant curve $\gamma$. In particular, it implies that the fibered Julia set of $P$ has zero-Lebesgue measure, so that \pref{thm:MRMTfibers} (MRMT depending on parameters) applies to bring $\Phi_\kappa$ as a continuous map in $\theta$ (see \textup{Case 2} in \textit{Step 3}), and thus a QPF map $P_\kappa$ at the end of the process. Observe that, even in the autonomous non-hyperbolic case, the Julia set needs not be of zero-measure \cite{Buff2012}.

    \noindent Nevertheless, without this assumption, although $\partial\mathcal{A}(\gamma)$ might present points of discontinuity for $\theta\mapsto \mu_{\kappa,\theta}^{}(z)$, the above surgery construction still provides a locally-defined fibered QC map to modify the fibered multiplier of an attracting invariant curve $\gamma$ for a given QPF holomorphic map (not necessarily a hyperbolic polynomial).
\end{remark}

\section{Quadratic family and Proof of Thm.~\ref{thm:B_QCapp}}
\label{sec:4_App}

In this final section we apply our QC surgery construction to change the fibered multiplier of a given attracting invariant curve for QPF hyperbolic polynomials within a generic one-parameter quadratic family.

Following the work of Sester \cite[\S2]{Sester1999} on fibered polynomials (over a continuous function on a given compact set), any QPF polynomial, of degree $d\geq 2$, is known to be affinely conjugate to a QPF unitary centered polynomial (of the same degree), reducing the parameter space to  $\mathcal{C}^*(\TT^1)\times \mathcal{C}(\TT^1)^{d-1}$. In general there may be a homotopic obstruction to make the fiber maps monic, depending on the degree of the polynomial and the homotopy class of its leading coefficient as a loop in $\CC^*$ (represented by its winding number about $0$); see \cite[Prop. 2.1, 2.2]{Sester1999}. However, it can be seen that there is no such an obstacle in the quadratic case.
    
As mentioned in \pref{sec:1_Intro}, here we choose to study the following parametrization of QPF quadratic polynomials:
\begin{equation}
    \label{eq:QuadraticS4}
    Q_\lambda(\theta,z) = \left(\theta+\alpha, \ z^2 + \lambda(\theta) z\right), \qquad \lambda\in\mathcal{C}^*(\TT^1),
\end{equation}
so that $\gamma\equiv 0$ is always an invariant curve, and the fiber maps $q_{\lambda,\theta}^{}:=\Pi_2\circ Q_\lambda$ satisfy that $q_{\lambda,\theta}'(0) = \lambda(\theta)$. Observe that $Q_\lambda$ is conjugate to the standard QPF quadratic map 
\begin{equation}
    \label{eq:StandardQuadratic}
    F_c(\theta,z):= \left(\theta+\alpha, \ z^2+c(\theta)\right), \qquad \mbox{with} \qquad c(\theta)=\lambda(\theta+\alpha)/2-\lambda(\theta)^2/4,
\end{equation}
via the fibered affine map $H(\theta,z)=\left(\theta, z+\lambda(\theta)/2\right)$ which sends the critical curve of $Q_\lambda$ to the zero-section.

Sester \cite[\S5]{Sester1999} described the \textit{connectedness locus} $\mathcal{M}_c(\TT^1)$ for the quadratic family $F_c$, i.e. those parameters $c\in\mathcal{C}(\TT^1)$ such that the fiber $\mathcal{K}_{c,\theta}$ (of the filled-in Julia set $\mathcal{K}_c$ of $F_c$) is connected for all $\theta$, showing first that
\begin{equation}
    \mathcal{M}_c(\TT^1)\subset \{ c: |c(\theta)|\leq 2 \}, \qquad \mbox{and} \qquad \mathcal{K}_{c,\theta}\subset \overline{B_2(0)}. 
\end{equation}
The \textit{generalized main cardioid} $\mathcal{M}_{c,0}(\TT^1)$ was defined as the subset of $\mathcal{M}_c(\TT^1)$ associated to hyperbolic polynomials $F_c$ such that $\mathrm{int}\hspace{0.3mm}\mathcal{K}_{c,\theta}$ is connected and non-empty for all $\theta$. As consequence of the characterization in \cite[Thm.~5.2]{Sester1999}, one obtains that $\mathcal{M}_{c,0}(\TT^1)$ is an open set in $\mathcal{C}(\TT^1)$ associated to QPF polynomials with a unique attracting invariant curve, including the parameters $\{c:|c(\theta)|<1/4\}$. 

Remarkably, as in the one-dimensional case, this main cardioid turns out to be in correspondence with those QPF quadratic polynomials for which each fiber $\mathcal{K}_{c,\theta}$ is a \textit{$k$-quasidisk} (with $k$ independent of $\theta$), i.e. the image of $\DD$ under a $k$-quasiconformal map; see more details in \cite[\S5]{Sester1999}. Denoting by $\mathcal{M}_0(\TT^1)$ the corresponding generalized main cardioid for the family $\mathcal{Q}_\lambda$, i.e. the hyperbolic polynomials $Q_\lambda$ such that all fibers of its filled-in Julia set $\mathcal{K}_\lambda$ have non-empty connected interior, these results can be restated as follows due to the dynamical relation between $F_c$ and $Q_\lambda$ (see also the role of critical curves in \pref{subsec:2_2_QPFpolyn}); indeed $\mathcal{K}_c=H(\mathcal{K}_\lambda)$.

\begin{lemma}[Main cardioid]
    \label{lem:Qfamily}
    Consider the family $\mathcal{Q}_\lambda$ of QPF quadratic polynomials in \pref{eq:QuadraticS4}, and let $\mathcal{M}_0(\TT^1)$ be described as above. If $\lambda\in\mathcal{M}_0(\TT^1)$, then $Q_\lambda$ has a unique attracting invariant curve $\gamma$, and its basin of attraction $\mathcal{A}(\gamma)$ is a connected set, containing the only critical curve $\omega(\theta):=-\lambda(\theta)/2$. Moreover, the fibers of $\mathcal{A}(\gamma)$ are quasidisks with
    \begin{equation}
        \mathcal{A}_\theta(\gamma)\subset \overline{B_2\left(\omega(\theta)\right)}.
    \end{equation}
\end{lemma}

Notice that $\mathcal{H}_0^*(\TT^1)$, the complex Banach manifold which is of our interest, defined by \pref{eq:MandelbrotQ0} as an open subset of $\mathcal{M}_0^*(\TT^1):=\mathcal{M}_0(\TT^1)\cap\mathcal{C}^*(\TT^1)$, corresponds to those QPF hyperbolic polynomials $Q_\lambda$ such that $\gamma\equiv 0$ is its unique attracting invariant curve, of index $m(\gamma) = \mathrm{wind}(\lambda(\theta),0)=0$ and Lyapunov exponent $\Lambda(\gamma) = \int_{\TT^1} \log|\lambda(\theta)|d\theta < 0$, with $\omega(\TT^1)\subset \mathcal{A}(\gamma)$.

In this setting we finally implement our QC surgery procedure (described in \pref{sec:3_ProofThmA}) to change the fibered multiplier of $\gamma\equiv 0$ for QPF hyperbolic polynomials within the quadratic family $\mathcal{Q}_\lambda$. This is done by means of an adequate normalization of the integrating maps resulting from the surgery construction, which provides a precise control of its only parameter $\lambda(\theta)$ in terms of the fibered multiplier. We ultimately show that the fibered multiplier map $\widehat{\kappa}:\mathcal{H}_0^*(\TT^1)\to\DD^*$ associated to $\gamma\equiv 0$ is a holomorphic submersion in the sense of \pref{subsec:2_4_Banach}.

\begin{customproof}{of \pref{thm:B_QCapp}}
    \textup{Consider the family of QPF quadratic polynomials (with Diophantine frequency $\alpha$):
    \begin{equation*}
        Q_\lambda(\theta,z)=\left(\theta+\alpha, \ z^2+\lambda(\theta) z\right),
    \end{equation*}
    where $\lambda\in\mathcal{H}_0^*(\TT^1)$, i.e. $Q_\lambda$ is hyperbolic with $\gamma\equiv 0$ as its unique attracting invariant curve, of zero-index and $\Lambda(\gamma)=\int_{\TT^1} \log|\lambda(\theta)|<0$. The associated fibered multiplier map $\widehat{\kappa}:\mathcal{H}_0^*(\TT^1)\to\DD^*$, is given by
    \begin{equation*}
       \widehat{\kappa}(\lambda) = \exp{\widehat{\chi}(\lambda)},  \qquad \mbox{where} \qquad \widehat{\chi}(\lambda):= \int_{\TT^1} \mathrm{log} \hspace{0.3mm} \lambda(\theta) d\theta
    \end{equation*}
    is a mapping from $\mathcal{H}_0^*(\TT^1)$ to the left half-plane $\HH_\ell$, acting as the complex version of the Lyapunov exponent. 
    }

    \noindent\textup{First we notice that $\widehat{\kappa}$ is a continuous map. In fact, for any $\lambda\in \mathcal{H}_0^*(\TT^1)$ and any sequence $\{\lambda_n\}_{n\in\NN}$ in $\mathcal{H}_0^*(\TT^1)$ converging to $\lambda$, i.e. $||\lambda_n-\lambda||_\infty\longrightarrow  0$ as $n\to\infty$, we have that
    \begin{equation*}
        \big|\widehat{\chi}(\lambda_n) - \widehat{\chi}(\lambda) \big| = \left| \int_{\TT^1} \log \hspace{0.1mm}\lambda_n(\theta) \hspace{0.2mm} d\theta - \int_{\TT^1} \log\hspace{0.1mm}\lambda(\theta)\hspace{0.2mm} d\theta  \right|\longrightarrow 0, \qquad \mbox{as} \qquad n\to\infty,
    \end{equation*}
    by the Dominated Convergence Theorem (see e.g. \cite[\S6]{Mujica1986}), since $\log\circ\vspace{0.2mm}\lambda_n$ is continuous on $\TT^1$ for each $n$. 
    }
    
    \noindent\textup{Moreover, as $\widehat{\kappa}$ is clearly bounded on $ \mathcal{H}_0^*(\TT^1)$, we can see that $\widehat{\kappa}=\exp\circ\widehat{\chi}$ is holomorphic by showing that it is weakly holomorphic (see \pref{def:HolomorphicBanach} and \pref{thm:WeaklyHoloPower}), i.e. for any $\lambda\in \mathcal{H}_0^*(\TT^1)$ and $v\in \mathcal{C}(\TT^1)$,
    \begin{equation*}
        t \mapsto \widehat{\chi}(\lambda+ t v)=\int_{\TT^1} \log\big(\lambda(\theta)+t \hspace{0.2mm} v(\theta)\big) d\theta
    \end{equation*}
    is holomorphic (as a function of one complex variable) on some neighborhood $\Delta$ of the origin, small enough such that $\lambda+t v\in\mathcal{C}^*(\TT^1)$. This follows from Morera's Theorem, as for any simple closed curve $\Gamma$ in $\Delta$, $\int_\Gamma \widehat{\chi}(\lambda+ t v) dt = 0$ by changing the order of integration via Fubini's Theorem and using Cauchy's Theorem. 
    }

    \noindent\textup{Now we can prove that $\widehat{\kappa}$ is a holomorphic submersion on $\mathcal{H}_0^*(\TT^1)$ by means of \pref{thm:HoloSubmersions}, that is, for each $\lambda^*\in\mathcal{H}_0^*(\TT^1)$, we shall find a local holomorphic section for $\widehat{\kappa}$ at $\kappa^*:=\widehat{\kappa}(\lambda^*)\in\DD^*$. Indeed, the holomorphic right-inverse of $\widehat{\kappa}$ may be determined on an arbitrary simply-connected neighborhood $U_{}^*\subset\DD^*$ of $\kappa^*$, by applying \pref{thm:A_QCqpf} in a manner that delivers the new polynomial in the same QPF holomorphic family.
    }

    \noindent\textup{Recall that, for any $\kappa\in U_{}^*$, \pref{thm:A_QCqpf} (see surgery construction in \pref{sec:3_ProofThmA}) provides a fibered QC map $\Phi_\kappa=(\mathrm{Id},\phi_{\kappa,\theta^{}})$ which is continuous in $\theta$ and holomorphic in $z$, conjugating $Q_{\lambda_{}^*}$ to some QPF hyperbolic polynomial (of the same degree) with $\Phi_\kappa(\gamma)$ as an attracting invariant curve of fibered multiplier $\kappa$, and zero-index too. 
    }

    \noindent\textup{For our purpose, the appropriate normalization of the integrating maps is the following: each $\phi_{\kappa,\theta}$ is chosen to fix $0$ and $\infty$, but instead of fixing $1$ as usual, here it is required that
    \begin{equation*}
        \frac{\phi_{\kappa,\theta}(z)}{z} \longrightarrow 1, \qquad \mbox{as} \qquad |z|\to \infty,
    \end{equation*}
    i.e. $\phi_{\kappa,\theta}$ is taken to be tangent to the identity at $\infty$, for all $\theta\in\TT^1$. In this manner, setting $u=\phi_{\kappa,\theta}(z)$ and taking into account that $q_{\lambda_{}^*,\theta}:=\Pi_2\circ Q_{\lambda_{}^*}$ is a monic quadratic polynomial, we obtain that
    \begin{equation*}
        \frac{\phi_{\kappa,\theta+\alpha}\circ q_{\lambda_{}^*,\theta}^{} \circ \phi_{\kappa,\theta}^{-1}(u)}{u^2} = \frac{\phi_{\kappa,\theta+\alpha}\big(q_{\lambda_{}^*,\theta}^{}(z)\big)}{q_{\lambda_{}^*,\theta}^{}(z)} \hspace{0.3mm} \frac{q_{\lambda_{}^*,\theta}^{}(z)}{z^2} \hspace{0.3mm} \left(\frac{z}{\phi_{\kappa,\theta}(z)}\right)^2 \longrightarrow 1, \qquad \mbox{as} \qquad |u|\to \infty.
    \end{equation*}
    In other words, the QPF hyperbolic polynomial $\Phi_\kappa\circ Q_{\lambda_{}^*} \circ \Phi_\kappa^{-1}$ is also monic, with the zero-section $\TT^1\times\{0\}$ as the corresponding attracting invariant curve. Hence it belongs to the same quadratic family, i.e. it is of the form $Q_{\lambda_\kappa}=(\mathcal{R}_\alpha,q_{\lambda_\kappa,\theta}^{})$ for some $\lambda_{\kappa}\in\mathcal{H}_0^*(\TT^1)$, where $q_{\lambda_\kappa,\theta}^{}(z)=z^2+\lambda_\kappa(\theta) z$, so that the following diagram commutes:
    }
    \begin{figure}[H]
    \centering
    \begin{tikzcd}[column sep=30pt,row sep=30pt]
    \CC \arrow[r,"q_{\lambda_{}^*,\theta}^{}"] \arrow[d,"\phi_{\kappa,\theta}^{}"'] &
    \CC \arrow[d,"\phi_{\kappa,\theta+\alpha}^{}"]
    \\
    \CC \arrow[r,"q_{\lambda_\kappa,\theta}^{}"'] &
    \CC
    \end{tikzcd}      
    \end{figure}
    
    \noindent\textup{Furthermore, as $\Phi_{\kappa}$ is a fibered (orientation-preserving) homeomorphism, it sends the critical curve of $Q_{\lambda_{}^*}$ to that of $Q_{\lambda_\kappa}$, i.e. $\phi_{\kappa,\theta}\big(-\lambda_{}^*(\theta)/2\big) = -\lambda_\kappa(\theta)/2$, and so the new parameter is uniquely determined by
    \begin{equation}
        \label{eq:B_parameters}
        \lambda_\kappa(\theta) = -2 \hspace{0.3mm}\phi_{\kappa,\theta} \big(\shortminus\lambda_{}^*(\theta)/2\big). 
    \end{equation}
    Therefore, the mapping $\widehat{\lambda}:U_{}^*\to\mathcal{H}_0^*(\TT^1)$, given by $\widehat{\lambda}(\kappa):=\lambda_\kappa$ as above, is a local right-inverse to the fibered multiplier map $\widehat{\kappa}$, i.e. $\widehat{\kappa}\circ\widehat{\lambda}=\mathrm{Id}_{U_{}^*}$, with $\widehat{\lambda}(\kappa^*)=\lambda_{}^*$. Since $\phi_{\kappa,\theta}$ is analytic in $\kappa$, it follows from relation \pref{eq:B_parameters} that $\widehat{\lambda}$ is holomorphic on $U_{}^*$. This concludes the proof of \pref{thm:B_QCapp}.
    $\hfill\square$ 
    }
\end{customproof}

%
\printbibliography
%
\end{document}